\documentclass[a4paper]{amsart}
\usepackage{amsmath,amsthm,amssymb,latexsym,epic,bbm,comment,mathrsfs}
\usepackage{graphicx,enumerate,stmaryrd,color,tikz,ytableau,xcolor}
\usepackage[all,2cell]{xy}
\xyoption{2cell}
\usetikzlibrary{patterns,snakes}

\newtheorem{theorem}{Theorem}
\newtheorem{lemma}[theorem]{Lemma}
\newtheorem{corollary}[theorem]{Corollary}
\newtheorem{proposition}[theorem]{Proposition}
\newtheorem{conjecture}[theorem]{Conjecture}
\newtheorem{example}[theorem]{Example}
\newtheorem{question}[theorem]{Question}
\newtheorem{remark}[theorem]{Remark}
\usepackage[all]{xy}
\usepackage[active]{srcltx}
\usepackage[parfill]{parskip}
\usepackage{enumerate}

\newcommand{\tto}{\twoheadrightarrow}

\begin{document}
\title[Hecke combinatorics and Kostant's problem]
{Hecke combinatorics,  K{\aa}hrstr{\"o}m's conditions  \\ and Kostant's problem}

\author[Samuel Creedon and Volodymyr Mazorchuk]
{Samuel Creedon and Volodymyr Mazorchuk}

\begin{abstract}
This paper discusses various aspects of the Hecke algebra 
combinatorics that are related to conditions appearing in 
K{\aa}hrstr{\"o}m's conjecture, {{}{}which} addresses Kostant's problem
for simple highest weight modules in the Bernstein-Gelfand-Gelfand 
category $\mathcal{O}$ for the complex Lie algebra $\mathfrak{sl}_n$. 
In particular, we study cyclic submodules of the regular Hecke 
module that are generated by the elements of the (dual) 
Kazhdan-Lusztig basis as well as the problem of left cell
invariance for both categorical and combinatorial 
K{\aa}hrstr{\"o}m's conditions.
\end{abstract}

\maketitle

\section{Introduction and description of results}\label{s0}

The Hecke algebra $\mathbf{H}$ of a finite Weyl group $W$
is an interesting object with many important applications 
in various areas of mathematics, including representation
theory. For example, it controls the combinatorics of the
regular block of the Bernstein-Gelfand-Gelfand (BGG)
category $\mathcal{O}$ associated to a semi-simple complex 
finite dimensional Lie algebra $\mathfrak{g}$ which has
$W$ as its Weyl group, see \cite{BGG,Hu,So0}.

The connection to category $\mathcal{O}$ is provided by the
so-called {\em Kazhdan-Lusztig (KL) combinatorics} of the algebra
$\mathbf{H}$, see \cite{KL}, which endows {{}{}$\mathbf{H}$} with a special
basis, called the {\em KL basis}, that has many
remarkable properties. One of these is that the structure
constants of $\mathbf{H}$ with respect to this basis are
non-negative (more precisely, are Laurent
polynomials with non-negative
coefficients). This non-negativity allows one to define 
three natural partial pre-orders on the elements of the
KL basis (alternatively, on the elements of $W$, {{}{}which} index
the elements of the KL basis), called the left, the right 
and the two-sided orders (see also \cite{KiMb} for a general
framework). The corresponding equivalence classes, called cells,
contain a lot of useful information. For example, they can be
used to construct various $\mathbf{H}$-modules.

In representation theory of Lie algebras, there is a classical
open problem, usually referred to as {\em Kostant's problem},
see \cite{Jo}. A few years ago, Johan K{\aa}hrstr{\"o}m
proposed a conjecture which relates Kostant's problem for
simple highest weight modules over $\mathfrak{sl}_n$ to
certain combinatorial questions for the algebra $\mathbf{H}$,
see \cite{KMM}. This conjecture led to significant progress in
understanding Kostant's problem, see \cite{Ma23} for an overview, 
including a  complete solution to Kostant's problem for simple 
highest weight modules indexed by fully commutative permutations, 
see \cite{MMM}. As a few very recent results, we could mention
our recent preprint \cite{CM1} where we prove that, 
asymptotically, the answer to Kostant's problem is negative,
for almost all simple highest weight modules. Also, in 
\cite{CM2} we completely solve Kostant's problem for simple
highest weight modules in the principal block of 
category $\mathcal{O}$, for the Lie algebra $\mathfrak{sl}_7$
(which was at the time of publication of \cite{CM2} the smallest
case for which the answer was not known).

Our attempt to understand K{\aa}hrstr{\"o}m's conjecture in the
general context of Kostant's problem led to some unexpected
revelations. It is known from \cite{MS1,MS2} that the answer to 
Kostant's problem is a left cell invariant. The original 
K{\aa}hrstr{\"o}m's conjecture is formulated only for 
involutions in the symmetric group, which form a cross-section
of the collection of all left cells. However, various inductive
methods which we used in \cite{CM2}, like parabolic induction,
heavily rely on the left cell invariance. Strangely enough, we
could use these methods to move around the information about
the answer to Kostant's problem, but we could not use them
to move around the information about the validity of the 
combinatorial conditions that appear in K{\aa}hrstr{\"o}m's conjecture.
This naturally led to the question whether the combinatorial
conditions which appear in K{\aa}hrstr{\"o}m's conjecture 
(where they were formulated for involutions only) are, in fact,
left cell invariants.

Trying to prove this left cell invariance, we faced the following
problem: Given an element $\underline{H}_x$ in the KL basis of
$\mathbf{H}$, consider the corresponding cyclic $\mathbf{H}$-module
$\mathbf{H}\underline{H}_x$. Let $\underline{H}_y$ be some
other element in the KL basis that is left equivalent to 
$\underline{H}_x$. Is it true that $\underline{H}_y\in 
\mathbf{H}\underline{H}_x$? It is very natural to expect that the
answer should be ``yes'' and that is exactly what we (as well as several
other colleagues whom we asked, and also ChatGPT5) thought.
It was very surprising for us to learn that this was a misconception.
The present paper arose as an attempt to understand what exactly
do we know about the combinatorics of  $\mathbf{H}$ related to 
various conditions that appear in K{\aa}hrstr{\"o}m's conjecture.

Here, one natural goal would be to prove left 
invariance of all conditions which appear in K{\aa}hrstr{\"o}m's 
conjecture. At the present stage, we cannot do that in full 
generality. K{\aa}hrstr{\"o}m's conjecture, see Subsection~\ref{s4.2}
for details, has conditions of different flavours: some
are categorical and some are combinatorial. It turns out that
the categorical conditions are ``easier'', mostly thanks to the 
fact that categorical arguments are based on a much richer structure
with various significant rigidity properties. We do prove 
the left cell invariance of the categorical 
K{\aa}hrstr{\"o}m's condition in Subsection~\ref{s4.4}.
The combinatorial conditions are much more ``difficult'',
and here we obtain only a few partial results.

We start our paper with a recap of basic combinatorics of
the Hecke algebra $\mathbf{H}$ in Section~\ref{s1}. The algebra
$\mathbf{H}$ has a number of distinguished bases, in particular,
the standard basis, the Kazhdan-Lusztig basis and the dual
Kazhdan-Lusztig basis. The interplay between these different 
bases is important for the rest of the paper. Therefore,
in Section~\ref{s1}, we present and recall a number of small
technical observations about the properties of these different bases.

The main application of $\mathbf{H}$ to representation theory that
we are interested in is within the study of BGG category $\mathcal{O}$.
Section~\ref{s2} is devoted to a detailed description of this connection.

In Section~\ref{s3} we study cyclic submodules of the regular 
$\mathbf{H}$-module that are generated by the elements of the
(dual) KL basis. Here our main results are:
\begin{itemize}
\item If $w_0'$ is the longest element of some parabolic subgroup
of the Weyl group and $\underline{H}_{w_0'}$ is the corresponding element
of the KL-basis, then $\mathbf{H}\underline{H}_{w_0'}$ contains
$\underline{H}_{w}$, for any $w$ that is greater than or equal to 
$w_0'$ with respect to the left KL order, see Theorem~\ref{thm-s3.3-4}.
\item Let $w_0'$ be the longest element of some parabolic subgroup
of the Weyl group and $w_0$ the longest element in the Weyl group.
Let $\underline{\hat{H}}_{w_0'w_0}$ be the element
of the dual KL-basis corresponding to $w_0'w_0$.
Then $\mathbf{H}\underline{\hat{H}}_{w_0'w_0}$ contains
$\underline{\hat{H}}_{w}$, for any $w$ that is smaller than or equal to 
$w_0'w_0$ with respect to the left KL order, see Proposition~\ref{prop-s3.4-4}.
\item An element $\underline{H}_{w}$ of a KL basis has the property that
$\underline{H}_{w}^2$ equals $\underline{H}_{w}$, up to a
scalar from the basis ring $\mathbb{Z}[v,v^{-1}]$, if and only if
$w$ is the longest element of some parabolic subgroup, see 
Proposition~\ref{prop-s3.3-5}.
\end{itemize}

In Section~\ref{s4} we study various properties of 
conditions that appear in K{\aa}hrstr{\"o}m's conjecture.
The conditions we study are of two different natures:
there are two conditions that are categorical and two conditions
that are combinatorial. The categorical condition  asks
whether, for $w\in S_n$, there exist different 
$x,y\in S_n$ such that $\theta_x L_w\cong \theta_y L_w\neq 0$.
Here $L_w$ is the simple module in the {{}{}
principal block $\mathcal{O}_0$ of the}
BGG category $\mathcal{O}$
corresponding to $w$ and $\theta_x$ and $\theta_y$ are 
indecomposable projective functors. This conditions also
comes in two forms: the graded version and the ungraded version,
depending on whether we view the graded or the ungraded version
of {{}{}$\mathcal{O}_0$}.

The combinatorial conditions are the shadows of the categorical
conditions in the algebra $\mathbf{H}$ (or the group algebra of
the Weyl group, respectively, in the ungraded case). So, the graded version
asks, for $w\in S_n$, whether there exist different 
$x,y\in S_n$ such that $\underline{\hat{H}}_w\underline{{H}}_x=
\underline{\hat{H}}_w\underline{{H}}_y\neq 0$.
The ungraded version is obtained by applying the evaluation
of the Hecke parameter $v$ {{}{}at} $1$ to the
graded version.

Our main results in Section~\ref{s4} are:
\begin{itemize}
\item We prove left cell invariance of the categorical 
conditions in Subsection~\ref{s4.4}. This left cell invariance
means that, if the (graded or ungraded) categorical 
K{\aa}hrstr{\"o}m's  condition is satisfied for some
$w$, with the corresponding $x$ and $y$, then it is satisfied 
for any $w'$ in the left cell of $w$, with the same $x$ and $y$.
\item In Subsection~\ref{s4.9} we prove that combinatorial 
K{\aa}hrstr{\"o}m's conditions are compatible with parabolic induction
(in the sense that they are satisfied for the input module of the
parabolic induction if and only if they are satisfied for the
output module).
\item In Subsection~\ref{s4.5}, we show that the elements
$x$ and $y$ that appear in K{\aa}hr\-str{\"o}m's conditions
must be in the same left cell.
\item In Subsection~\ref{s4.6}, we establish the following
property of the categorical K{\aa}hr\-str{\"o}m's conditions:
if the are  satisfied for some $w$ with the corresponding
$x$ and $y$, then one can replace $x$ and $y$ by any
$x'$ and $y'$, respectively, such that 
$x$ and $x'$ are in the same right cell, 
$y$ and $y'$ are in the same right cell and 
$x'$ and $y'$ are in the same left cell. 
\end{itemize}
In Subsection~\ref{s4.7} we collected some partial results
towards the left cell invariance for the combinatorial
conditions. During the whole paper, we formulate several open
problems and {{}{}conjectures}.

\subsection*{Acknowledgments} This research is supported
by Vergstiftelsen. The first author is also partially
supported by the Swedish Research Council. Some of the results
of this paper were presented at the conference
``Representation theory down under'' which was organized at 
the University of Sydney in June 2025. We thank the organizers
for this opportunity and support. {{}{}We thank the referee
for helpful comments.}

\section{Hecke algebra}\label{s1}

\subsection{Basic definitions}\label{s1.1}

Let $\mathbb{A}:=\mathbb{Z}[v,v^{-1}]$. For $a\in \mathbb{A}$
and $i\in\mathbb{Z}$, we denote by $(a:i)$ the
coefficient at $v^{i}$ in $a$, that is 
\begin{displaymath}
a=\sum_{i\in\mathbb{Z}} (a:i)v^i.
\end{displaymath}

Let $W$ be a {\em finite Weyl group} with a fixed set $S$
of simple reflections. The corresponding
{\em Hecke algebra} $\mathbf{H}=\mathbf{H}_W$ is defined
as the $\mathbb{A}$-algebra generated
by $H_s$, where $s\in S$, subject to the same braid 
relations as for the Coxeter system $(W,S)$ 
together with the quadratic  relations $H_s^2=H_e+(v^{-1}-v)H_s$,
for each $s\in S$ (see \cite{KL}, we use the normalization
of \cite{So}). The algebra $\mathbf{H}$ has 
the {\em standard basis} $\{H_w\,:\,w\in W\}$. Here
$H_w=H_{s_1}H_{s_2}\dots H_{s_k}$ provided that
$w=s_1s_2\dots s_k$ is a reduced expression.
Directly from the definitions we see that 
all $H_s$ and hence all $H_w$ as well 
are invertible elements of $\mathbf{H}$.

Let $({}_-)|_{{}_{v=1}}:\mathbb{A}\rightarrow\mathbb{Z}$ denote 
the surjective ring homomorphism defined by evaluating $v$ 
{{}{}at $1$}, and the ring obtained from $\mathbf{H}$ by extending 
scalars through $({}_-)|_{{}_{v=1}}$ by
\[ \mathbf{H}^{\mathbb{Z}}:=\mathbb{Z}\otimes_{\mathbb{A}}\mathbf{H}. \]
For $H\in\mathbf{H}$, we abuse notation and also write $H$ for 
the element $H|_{{}_{v=1}}=1\otimes H\in\mathbf{H}^{\mathbb{Z}}$. 
We similarly let $({}_-)|_{{}_{v=1}}$ denote the canonical ring 
epimorphism $\mathbf{H}\rightarrow\mathbf{H}^{\mathbb{Z}}$. 
Clearly, the kernel of $({}_-)|_{{}_{v=1}}$ is spanned by elements 
$aH_{w}$ with $a\in\mathbb{A}$ such that $a|_{{}_{v=1}}=0$.

The {{}{}$\mathbb{Z}$-linear} {\em bar involution} on $\mathbf{H}$ is defined 
via $\overline{H_s}=H_s^{-1}$ and $\overline{v}=v^{-1}$.
The algebra $\mathbf{H}$ has the {\em Kazhdan-Lusztig (KL)}
basis $\{\underline{H}_w\,:\,w\in W\}$ which is uniquely
determined by the following properties:
\begin{itemize}
\item $\overline{\underline{H}_w}=\underline{H}_w$;
\item $\displaystyle \underline{H}_w\,
{{}{}\in}\,
H_w+\sum_{x\in W}v\mathbb{Z}[v]H_x$.
\end{itemize}
We have 
\begin{equation}\label{eq-KLvsSt}
\underline{H}_x=\sum_{y\in W}h_{x,y}H_y, 
\end{equation}
where $h_{x,y}\in \mathbb{Z}[v]$ are the {\em Kazhdan-Lusztig polynomials}.

Denote by $\prec$ the Bruhat order on $W$ and by
$\ell$ the length function with respect to $S$.
Then $h_{x,y}(v)\neq 0$ implies $y\preceq x$.
Moreover, if $h_{x,y}(v)\neq 0$, then the
degree  of $h_{x,y}(v)$ equals $\ell(x)-\ell(y)$
and the leading coefficient equals $1$. 
For $y\preceq x$, let $\mu(y,x)$ denote the 
coefficient at $v$ in $h_{x,y}(v)$, that is
$(h_{x,y}(v):1)$. We also set $\mu(x,y):=\mu(y,x)$.
In all other cases, we set $\mu(x,y)=0$.
This is  the {\em Kazhdan-Lusztig $\mu$-function}.

\subsection{Structure constants}\label{s1.2}

We denote by {{}{}$\mathbf{h}_{x,y}^w\in \mathbb{A}$}, where $x,y,w\in W$, 
the structure constants of $\mathbf{H}$ with respect 
to the Kazhdan-Lusztig basis, that is:
\begin{equation}\label{eq-strconst}
\underline{H}_x\cdot \underline{H}_y
=\sum_{w}{{}{}\mathbf{h}_{x,y}^w}\underline{H}_w.
\end{equation}
Then all {{}{}$\mathbf{h}_{x,y}^w$} are bar-invariant and
have positive coefficients. For a simple
reflection $s$ and $y\in W$, we have
\begin{displaymath}
\underline{H}_s\cdot \underline{H}_y=
\begin{cases}
(v+v^{-1}) \underline{H}_y,& sy\prec y;\\
\displaystyle\sum_{w:sw\prec w}\mu(y,w)\underline{H}_w,& y\prec sy.
\end{cases}
\end{displaymath}

\subsection{Dual basis}\label{s1.3}

Let $\boldsymbol{\tau}:\mathbf{H}\to \mathbb{A}$ be the 
$\mathbb{A}$-linear map defined via 
$\boldsymbol{\tau}(H_e)=1$ and $\boldsymbol{\tau}(H_w)=0$, for $w\neq e$.
Define the {\em dual Kazhdan-Lusztig (KL) basis}
$\{\underline{\hat{H}}_w\,:\,w\in W\}$ of $\mathbf{H}$ via
$\boldsymbol{\tau}(\underline{\hat{H}}_x\underline{{H}}_{y^{-1}})
=\delta_{x,y}$
(the Kronecker $\delta$).
We can also define a non-degenerate symmetric
bilinear form
$({}_-,{}_-)$ on $\mathbf{H}$ with
values in $\mathbb{A}$ via
$(a,b):=\boldsymbol{\tau}(ab)$.
Then, directly from the definitions, 
the KL and the dual KL bases are dual
to each other with respect to this form.

We will also let $\boldsymbol{\tau}_{v=1}:\mathbf{H}^{\mathbb{Z}}\to \mathbb{Z}$ be the $\mathbb{Z}$-linear map defined by $\boldsymbol{\tau}_{v=1}(H_{e})=1$ and $\boldsymbol{\tau}_{v=1}(H_{w})=0$ for $w\neq e$. One can similarly define a non-degenerate symmetric bilinear form on $\mathbf{H}^{\mathbb{Z}}$ using $\boldsymbol{\tau}_{v=1}$, 
{{}{}such that} 
the corresponding KL and dual KL basis elements of $\mathbf{H}^{\mathbb{Z}}$ (i.e. the projections of those in $\mathbf{H}$) are indeed dual to one another.

\subsection{Ring involutions}\label{s1.3.5}

Let $\boldsymbol{\beta}:\mathbf{H}\rightarrow\mathbf{H}$ 
be the ring involution sending $v$ to $-v^{-1}$ and fixing 
each $H_w$. Moreover, let $({}_-)^{\ast}:\mathbf{H}\rightarrow\mathbf{H}$ 
denote the ring anti-automorphism which fixes the parameter 
$v$ and sends $H_{w}$ to $H_{w^{-1}}$. We have the following properties:
\begin{equation}\label{lem-InvBasicProps}
(\underline{H}_{w})^{\ast}=\underline{H}_{w^{-1}}, \hspace{2mm} \boldsymbol{\beta}\circ({}_-)^{\ast}=({}_-)^{\ast}\circ\boldsymbol{\beta}, \ \text{ and } \ \boldsymbol{\tau}=\boldsymbol{\tau}\circ(-)^{\ast}.
\end{equation}
The first can be deduced from the defining properties of 
the KL basis, and the latter two follow from {{}{}the} definitions. 
By \cite[Chapter~5]{Lu} (see also \cite{Vi}), for all $w\in W$, we have
\begin{equation}\label{eq-virk}
\underline{\hat{H}}_{ww_0}=
\boldsymbol{\beta}(\underline{H}_wH_{w_0})
=\boldsymbol{\beta}(\underline{H}_w)H_{w_0}.
\end{equation}
We also define ring involutions 
$\boldsymbol{\beta}_{v=1}:\mathbf{H}^{\mathbb{Z}}\rightarrow\mathbf{H}^{\mathbb{Z}}$ 
and $({}_-)^{\ast}_{v=1}:\mathbf{H}^{\mathbb{Z}}\rightarrow\mathbf{H}^{\mathbb{Z}}$ 
by
\[ \boldsymbol{\beta}_{v=1}(h|_{{}_{v=1}}):=\boldsymbol{\beta}(h)|_{{}_{v=1}} \ \text{ and } \ (h|_{{}_{v=1}})_{v=1}^{\ast}:=(h)^{\ast}|_{{}_{v=1}}, \]
for any $h\in\mathbf{H}$. One has to show that 
both of these maps are well-defined, which comes 
down to showing that, for all $k$ in the kernel of 
the projection $({}_-)|_{{}_{v=1}}:\mathbf{H}\rightarrow\mathbf{H}^{\mathbb{Z}}$, 
we have $\boldsymbol{\beta}(k)|_{{}_{v=1}}=0$ and $(k)^{\ast}|_{{}_{v=1}}=0$. 
This follows by noting that both $\boldsymbol{\beta}$ and 
$({}_-)^{\ast}$ preserve the property that $a|_{{}_{v=1}}=0$ for 
any element $a\in\mathbb{A}$. Therefore, one can check 
that all the relations from Equations~\eqref{lem-InvBasicProps} 
and \eqref{eq-virk} hold if one replaces each involution with 
its $v=1$ counterpart and each basis element with its 
projection in $\mathbf{H}^{\mathbb{Z}}$.

\subsection{Dual structure constants}\label{s1.4}

We denote by {{}{}$\hat{\mathbf{h}}_{x,y}^w\in \mathbb{A}$,} where $x,y,w\in W$, 
the structure constants of $\mathbf{H}$ with respect 
to the dual Kazhdan-Lusztig basis, that is:
\begin{equation}\label{eq-dstrconst}
\underline{\hat{H}}_x\cdot \underline{\hat{H}}_y
=\sum_{w}{{}{}\hat{\mathbf{h}}_{x,y}^w}\underline{\hat{H}}_w.
\end{equation}

\subsection{Kazhdan-Lusztig orders and cells}\label{s1.5}

For $x,y\in W$, we write $x\leq_L y$ provided that
there is $w\in W$ such that $\underline{H}_y$ appears with a
non-zero coefficient in the expression of $\underline{H}_w\underline{H}_x$
with respect to the KL basis. The relation $\leq_L$ is a pre-order
called the {\em left KL pre-order}. The equivalence classes
for $\leq_L$ are called {\em left cells} and the corresponding
equivalence relation is denoted $\sim_L$. 
The {\em right pre-order} $\leq_R$, right equivalence relation
$\sim_R$ and {\em right cells} are defined similarly using
$\underline{H}_x\underline{H}_w$ instead of 
$\underline{H}_w\underline{H}_x$.
The {\em two-sided pre-order} $\leq_J$, two-sided equivalence relation
$\sim_J$ and {\em two-sided cells} are defined similarly using
$\underline{H}_{w_1}\underline{H}_x\underline{H}_{w_2}$ instead of 
$\underline{H}_w\underline{H}_x$.

\subsection{Kazhdan-Lusztig orders and cells for symmetric groups}\label{s1.6}

Consider the case $W=S_n$ with the set of simple reflections given by the
elementary transpositions. We let $\Lambda_{n}$ denote the set of all 
Young diagrams of size $n$, and we let {{}{}$\triangleleft$} 
denote the usual dominance 
order on $\Lambda_{n}$. The classical Robinson-Schensted
correspondence, see \cite{Sch,Sa}, assigns to each element of $S_n$
a pair of standard Young tableaux of the same shape (the shape itself
is a partition of $n$):
\begin{displaymath}
\mathbf{RS}:S_n\to\coprod_{{\lambda\in\Lambda_{n}}}
\mathbf{SYT}_\lambda\times \mathbf{SYT}_\lambda.
\end{displaymath}
For $w\in S_n$, we write $\mathbf{RS}(w)=(\mathtt{P}_w,\mathtt{Q}_w)$,
where $\mathtt{P}_w$ is the insertion tableau and 
$\mathtt{Q}_w$ is the recording tableau. 
We also write $\mathbf{sh}(w):=\lambda$ if $\lambda$ 
is the shape of $\mathtt{P}_{w}$ (equivalently $\mathtt{Q}_{w}$), 
and we refer to $\lambda$ as the \emph{shape} of $w$. By \cite[\S~5]{KL}, 
for $x,y\in S_n$, we have
\begin{itemize}
\item $x\sim_L y$ if and only if $\mathtt{Q}_x=\mathtt{Q}_y$,
\item $x\sim_R y$ if and only if $\mathtt{P}_x=\mathtt{P}_y$,
\item $x\sim_J y$ if and only if $\mathbf{sh}(x)=\mathbf{sh}(y)$.
\end{itemize}
{{}{}For every $w\in S_n$, we have  $\mathtt{P}_w=\mathtt{Q}_{w^{-1}}$
and $\mathtt{Q}_w=\mathtt{P}_{w^{-1}}$. This immediately implies that 
$w\in S_n$ is an involution if and only if $\mathtt{P}_w=\mathtt{Q}_w$.
}
Each left (and each right) cell of $S_n$ contains a unique
involution. Moreover, the restrictions of $\leq_L$ and $\leq_R$
to the set of all involutions of $S_n$ coincide.
In particular, there is a natural
bijection between involutions in $S_n$ and the set
\begin{displaymath}
\mathbf{SYT}_{n}:=\coprod_{{\lambda\in\Lambda_{n}}}
\mathbf{SYT}_\lambda. 
\end{displaymath}
Furthermore, by \cite[Theorem~5.1]{Ge}, we have
\begin{itemize}
\item $x\leq_L y$ implies {{}{}$\mathbf{sh}(y)\triangleleft\mathbf{sh}(x)$
(the dominance order)},
\item $x\leq_R y$ implies {{}{}$\mathbf{sh}(y)\triangleleft\mathbf{sh}(x)$}.
\end{itemize} 

\subsection{Lusztig's $\mathbf{a}$-function}\label{s1.7}

In \cite{Lu2}, Lusztig introduced the so-called
{\em $\mathbf{a}$-function} 
\begin{displaymath}
\mathbf{a}:W\to \mathbb{Z}_{\geq 0}.
\end{displaymath}
For $w\in W$, the value $\mathbf{a}(w)$ is the maximal degree
of {{}{}$\mathbf{h}_{w,x}^y$,} taken over all $x,y\in W$.
{{}{}For $w\in W$ with $\mathbf{sh}(w)=(\lambda_1,\dots,\lambda_k)$,
we have $\displaystyle \mathbf{a}(w)=\sum_{i=1}^k(i-1)\lambda_i$.
}
This function
has the following properties, see \cite[Section~13]{Lu3}:
\begin{itemize}
\item for $x,y\in W$ such that $x\leq_L y$, we have
$\mathbf{a}(x)\leq \mathbf{a}({{}{}y})$,
and similarly for $\leq_R$ and $\leq_J$;
\item for $x,y\in W$ such that $x<_L y$, we have
$\mathbf{a}(x)< \mathbf{a}({{}{}y})$,
and similarly for $<_R$ and $<_J$;
\item $\mathbf{a}(x)\leq\ell(x)$, for all $x\in W$;
\item $\mathbf{a}(x)=\ell(x)$ if $x$ is the
longest element of some parabolic subgroup of $W$;
\item if the degree of {{}{}$\mathbf{h}_{w,x}^y$} equals
$\mathbf{a}(w)$, then $x\sim_J w$.
\end{itemize}

\section{Categorification}\label{s2}

\subsection{Category $\mathcal{O}$}\label{s2.1}

Let $\mathfrak{g}$ be a semi-simple complex Lie algebra
with a fixed Cartan subalgebra $\mathfrak{h}$ such that
$W$ is the Weyl group of $(\mathfrak{g},\mathfrak{h})$.
To our chosen set $S$ of simple reflections, one associates
a triangular decomposition of $\mathfrak{g}$:
\begin{displaymath}
\mathfrak{g}=\mathfrak{n}_-\oplus \mathfrak{h}\oplus \mathfrak{n}_+. 
\end{displaymath}
Let $\mathcal{O}$ be the corresponding BGG category
of $\mathfrak{g}$-modules, see \cite{BGG,Hu}.
Let $\mathcal{O}_0$ be the {\em principal block} of
$\mathcal{O}$, that is, the direct summand containing the trivial
module. 

The isomorphism classes of simple objects in $\mathcal{O}_0$ are
indexed (bijectively) by the elements of $W$. We denote by
$L_w$ the simple object corresponding to $w$. Then $L_w$
is the simple highest weight $\mathfrak{g}$-module with highest weight
$w\cdot 0$, where $\cdot$ denotes the dot-action of $W$ on
$\mathfrak{h}^*$, that is $w\cdot \lambda:=w(\lambda+\rho)-\rho$,
where $\rho$ is the {{}{}half-sum} of all positive roots.
We denote by $\Delta_w$ the Verma cover of $L_w$. 

The category $\mathcal{O}_0$ is equivalent to the category of 
modules over a (basic) finite dimensional associative algebra, 
call it $A$. In particular, $\mathcal{O}_0$ has enough projectives.
We denote by $P_w$ the indecomposable projective cover of $L_w$
and by $I_w$ the indecomposable injective envelope of $L_w$.
The algebra $A$ is isomorphic to $A^{\mathrm{op}}$ via an 
isomorphism that fixes some complete set of primitive idempotents.
We denote by $\star$ the corresponding simple preserving
duality on $A$-mod. In particular, $(P_w)^\star\cong I_w$.
This allows us to define the {\em dual Verma modules}
$\nabla_w:=(\Delta_w)^\star$.

The category $\mathcal{O}_0$ is a highest weight category
in the sense of \cite{CPS}, equivalently, $A$ is a quasi-hereditary
algebra, see \cite{DR}. Moreover, this quasi-hereditary 
structure is, essentially, unique by \cite{Co}. 
We denote by $T_w$ the indecomposable tilting module
corresponding to $L_w$, whose existence follows from \cite{Ri}.
One important consequence of $A$ being quasi-hereditary
is that it has finite global dimension.

The category $\mathcal{O}_0$ is equipped with the action of the
monoidal category $\mathscr{P}$ of projective endofunctors,
see \cite{BG}. The underlying category $\mathscr{P}$ is additive,
idempotent split and has finite dimensional morphism spaces.
Indecomposable objects in $\mathscr{P}$ are in bijection with 
elements of $W$. For $w\in W$, we denote by $\theta_w$ the unique
indecomposable projective functors with the property 
$\theta_w P_e\cong P_w$. 

\subsection{Grading}\label{s2.2}

The algebra $A$ is Koszul, see \cite{So0}, and hence admits
a Koszul $\mathbb{Z}$-grading. The corresponding category of
finite dimensional $\mathbb{Z}$-graded $A$-module is denoted
${}^{\mathbb{Z}}\mathcal{O}_0$. Homomorphisms in 
${}^{\mathbb{Z}}\mathcal{O}_0$ are homogeneous module
homomorphisms of degree $0$.

All structural modules
in $\mathcal{O}_0$ admit graded lifts (unique up to isomorphism
and shift of grading). We will denote the {\em standard}
graded lifts of modules by the same symbols as the ungraded
notation. Thus $L_w$ is concentrated in degree $0$,
both $P_w$ and $\Delta_w$ have top in degree $0$,
while $I_w$ and $\nabla_w$ have socle in degree $0$.
The standard graded lift of $T_w$ is defined such that the
embedding $\Delta_w\hookrightarrow T_w$ is a morphism in
${}^{\mathbb{Z}}\mathcal{O}_0$.

The action of $\mathscr{P}$ on ${}^{\mathbb{Z}}\mathcal{O}_0$
also admits a graded lift. We denote by 
${}^{\mathbb{Z}}\mathscr{P}$ the corresponding monoidal 
category of graded projective functors. The standard graded
lift of $\theta_w$ is defined such that the isomorphism
$\theta_w P_e\cong P_w$ is an isomorphism in 
${}^{\mathbb{Z}}\mathcal{O}_0$.

The duality $\star$ admits a graded lift, which we denote by the
same symbol. We note that the graded version of $\star$ matches
the degree $i$ with the degree $-i$. In particular, 
the graded duality $\star$ is no
longer simple preserving, in general. We do have
$L_w^\star\cong L_w$ (since $L_w$ is concentrated in degree $0$), 
but then $(L_w\langle 1\rangle)^\star\cong
L_w\langle -1\rangle$.
We refer to \cite{St0} for further details.

\subsection{Decategorification}\label{s2.3}

Consider the Grothendieck group 
$\mathbf{Gr}({}^{\mathbb{Z}}\mathcal{O}_0)$
of the abelian category ${}^{\mathbb{Z}}\mathcal{O}_0$.
We can view it as an $\mathbb{A}$-module, where
$v^{-1}$ corresponds to the action of $\langle 1\rangle$.
Similarly, the split Grothendieck ring 
$\mathbf{Gr}_\oplus({}^{\mathbb{Z}}\mathscr{P})$
of the additive category ${}^{\mathbb{Z}}\mathscr{P}$
is an $\mathbb{A}$-module as well.

Fix an isomorphism $\Phi$ between 
$\mathbf{Gr}({}^{\mathbb{Z}}\mathcal{O}_0)$
and $\mathbf{H}$ which sends $[\Delta_w]$ to $H_w$,
for $w\in W$. Similarly, fix an isomorphism 
$\Psi$ between $\mathbf{Gr}_\oplus({}^{\mathbb{Z}}\mathscr{P})$
and $\mathbf{H}$ which sends $[\theta_w]$ to $\underline{H}_w$,
for $w\in W$. Consider the right regular action of
$\mathbf{H}$ (on itself).
Each $\theta\in {}^{\mathbb{Z}}\mathscr{P}$
induces a linear operator $[\theta]$ on
${}^{\mathbb{Z}}\mathcal{O}_0$ such that the
following diagram commutes:
\begin{displaymath}
\xymatrix{
\mathbf{Gr}({}^{\mathbb{Z}}\mathcal{O}_0)\ar[rr]^{[\theta]}
\ar[d]_{\Phi}
&&\mathbf{Gr}({}^{\mathbb{Z}}\mathcal{O}_0)\ar[d]_{\Phi}\\
\mathbf{H}_\mathbf{H}\ar[rr]^{\Psi([\theta])}&&\mathbf{H}_\mathbf{H}} 
\end{displaymath}
We can phrase this as follows: the action of 
${}^{\mathbb{Z}}\mathscr{P}$ on
${}^{\mathbb{Z}}\mathcal{O}_0$
decategorifies to the right regular action of 
$\mathbf{H}$. An immediate consequence of this is 
that $\Phi([P_w])=\underline{H}_w$, for $w\in W$.

We also have yet another basis of $\mathbf{H}$ given by 
$\{[T_w]\,:\, w\in W\}$.  Due to Ringel self-duality of 
$\mathcal{O}_0$, which is proved using the twisting functor
$\top_{w_0}$, see \cite{So3}, for $w\in W$, we have
\begin{equation}\label{tilting-basis}
[T_w]=H_{w_0}[P_{w_0w}]=H_{w_0}\underline{H}_{w_0w}. 
\end{equation}
Because of this, the basis $\{[T_w]\,:\, w\in W\}$
is usually called the {\em twisted Kazhdan-Lusztig basis}.

\subsection{Categorification of $({}_-,{}_-)$}\label{s2.4}

Consider the bounded derived category 
$\mathcal{D}^b({}^{\mathbb{Z}}\mathcal{O}_0)$.
Since $A$ has finite global dimension, the 
Grothendieck group of $\mathcal{D}^b({}^{\mathbb{Z}}\mathcal{O}_0)$
is canonically isomorphic to 
$\mathbf{Gr}({}^{\mathbb{Z}}\mathcal{O}_0)$.
As all functors in ${}^{\mathbb{Z}}\mathscr{P}$
are exact, they give rise to endofunctor of 
$\mathcal{D}^b({}^{\mathbb{Z}}\mathcal{O}_0)$
and their classes in 
$\mathbf{Gr}_\oplus({}^{\mathbb{Z}}\mathscr{P})$
are, naturally, linear operators on 
$\mathbf{Gr}(\mathcal{D}^b({}^{\mathbb{Z}}\mathcal{O}_0))$.
Therefore we also have the
following commutative diagram:
\begin{displaymath}
\xymatrix{
\mathbf{Gr}(\mathcal{D}^b({}^{\mathbb{Z}}\mathcal{O}_0))
\ar[rr]^{[\theta]}
\ar[d]_{\Phi}
&&\mathbf{Gr}(\mathcal{D}^b({}^{\mathbb{Z}}\mathcal{O}_0))\ar[d]_{\Phi}\\
\mathbf{H}_\mathbf{H}\ar[rr]^{\Psi([\theta])}&&\mathbf{H}_\mathbf{H}
} 
\end{displaymath}
For $M,N\in \mathcal{D}^b({}^{\mathbb{Z}}\mathcal{O}_0)$, the expression
\begin{equation}\label{eqnn1}
(([M],[N])):=
\sum_{i,j\in\mathbb{Z}}(-1)^{i}v^j\dim 
\mathcal{D}^b({}^{\mathbb{Z}}\mathcal{O}_0)(M[-i]\langle j\rangle,N^\star)
\end{equation}
defines a symmetric bilinear form on 
$\mathbf{Gr}(\mathcal{D}^b({}^{\mathbb{Z}}\mathcal{O}_0))$
with values in $\mathbb{A}$.

Note that, for $M,N\in {}^{\mathbb{Z}}\mathcal{O}_0$, all terms
with $i<0$ in \eqref{eqnn1} vanish. The homological orthogonality
of Verma and dual Verma modules gives
$(([\Delta_x],[\nabla_y]))=\delta_{x,y}$. This implies that
the above form corresponds, under $\Phi$, to the form $({}_-,{}_-)$
on $\mathbf{H}$, see Subsection~\ref{s1.3}. Now, the 
orthogonality of projective and simple modules implies
that $\Phi([L_w])=\underline{\hat{H}}_w$.

\subsection{Interpretation of structure constants}\label{s2.5}

As usual, for a module $M$ and a simple module $L$,
we denote by $[M:L]$ the composition multiplicity of 
$L$ in $M$. Recall the structure constants {{}{}$\mathbf{h}_{x,y}^w$} from 
\eqref{eq-strconst}.

\begin{proposition}\label{prop-s2.5-1}
For $x,y,w\in W$ and $k\in\mathbb{Z}$, we have
\begin{displaymath}
({{}{}\mathbf{h}_{x,y}^w}:k) = [\theta_{x^{-1}}L_w:L_y\langle k\rangle].
\end{displaymath}
\end{proposition}

\begin{proof}
Evaluating  \eqref{eq-strconst} at $P_e$, we get 
\begin{displaymath}
\theta_x\theta_y P_e\cong \sum_{w\in W} \sum_{k\in\mathbb{Z}}
P_w\langle -k\rangle^{\oplus ({{}{}\mathbf{h}_{x,y}^w}:k)}.
\end{displaymath}
The summand $P_w\langle -k\rangle^{\oplus ({{}{}\mathbf{h}_{x,y}^w}:k)}$
can be singled out by computing homomorphisms into $L_w\langle -k\rangle$.
Using adjunction, we obtain
\begin{displaymath}
({{}{}\mathbf{h}_{x,y}^w}:k)=\dim
\mathrm{hom}(\theta_x\theta_y P_e,L_w\langle -k\rangle) =
\dim
\mathrm{hom}(\theta_y P_e,\theta_{x^{-1}}L_w\langle -k\rangle).
\end{displaymath}
The claim follows.
\end{proof}

We note that, as the module $\theta_{x^{-1}}L_w$ is self-dual
with respect to $\star$, we have 
\begin{displaymath}
[\theta_{x^{-1}}L_w:L_y\langle k\rangle]=
[\theta_{x^{-1}}L_w:L_y\langle -k\rangle].
\end{displaymath}

\subsection{Soergel's character formula for tilting modules}\label{s2.6}

The indecomposable tilting module $T_w$ in
$\mathcal{O}_0$, where $w\in W$, is uniquely determined by the properties that
\begin{itemize}
\item $T_w^\star\cong T_w$;
\item $\Delta_w\hookrightarrow T_w$ such that the cokernel has
a filtration with Verma subquotients (alternatively, 
$T_w\tto \nabla_w$ such that the kernel has a filtration with
dual Verma subquotients).
\end{itemize}
We refer to \cite{CI,Ri} for details. The subquotients of the
Verma filtration of $T_w$ were determined by Soergel in 
\cite[Theorem~6.7]{So3}. For $x,y\in W$ and $k\in\mathbb{Z}$, we have:
\begin{displaymath}
[T_{w_0x}:\nabla_{w_0y}\langle k\rangle]=[P_x:\Delta_y\langle k\rangle]. 
\end{displaymath}
From the self-duality of {{}{}tilting} modules we obtain
\begin{displaymath}
[T_{w_0x}:\Delta_{w_0y}\langle k\rangle]=
{{}{}[T_{w_0x}:\nabla_{w_0y}\langle -k\rangle]}
=[P_x:\Delta_y\langle -k\rangle]. 
\end{displaymath}
As $T_{w_0x}=\theta_x T_{w_0}=\theta_x\Delta_{w_0}$, it follows
that
\begin{equation}\label{eq-tilt}
H_{w_0}\underline{H}_x=
\sum_{y\in W} h_{x,y}(v^{-1})H_{w_0y},
\end{equation}
{{}{}where $h_{x,y}$ are the KL-polynomials and
$h_{x,y}(v^{-1})$ denotes the evaluation of $h_{x,y}$ at $v^{-1}$
instead of $v$.
}

\subsection{Structure constants vs dual structure constants}\label{s2.7}

Using  \eqref{eq-tilt} and \eqref{eq-virk}, we can now 
determine the connection between the structure constants and the
dual structure constants.

\begin{proposition}\label{prop-s2.7-1}
For $x,y,w\in W$, we have
\begin{displaymath}
{{}{}\hat{\mathbf{h}}_{x,y}^w}=
\sum_{a,z\in W}(-1)^{\ell(w_0a)-\ell(z)}
\boldsymbol{\beta}(
h_{w_0z,a} h_{yw_0,a}(v^{-1})
{{}{}\mathbf{h}_{xw_0,z}^{ww_0}}).
\end{displaymath}
\end{proposition}

\begin{proof}
From \eqref{eq-virk}, we have
\begin{equation}\label{eq-prop-s2.7-1-1}
\underline{\hat{H}_x}\underline{\hat{H}_y}=
\boldsymbol{\beta}
(\underline{H}_{xw_0}H_{w_0}\underline{H}_{yw_0}H_{w_0}).
\end{equation}
Using \eqref{eq-tilt}, we can rewrite
$\underline{H}_{xw_0}H_{w_0}\underline{H}_{yw_0}H_{w_0}$
as
\begin{equation}\label{eq-prop-s2.7-1-2}
\sum_{a\in W} h_{yw_0,a}(v^{-1})\underline{H}_{xw_0}H_{w_0a} H_{w_0}.
\end{equation}

Recall from \cite[(5.1.3)]{Lu}
or \cite[Theorem~3.1]{KL} that, for $u,v\in W$, we have
\begin{displaymath}
\delta_{u,v}=
\sum_{u\preceq  w\preceq  v}(-1)^{\ell(v)-\ell(w)}h_{u,w}h_{w_0v,w_0w}.
\end{displaymath}
Consequently, for $a\in W$, we have
\begin{displaymath}
H_{w_0a}=\sum_{z\preceq w_0a}(-1)^{\ell(w_0a)-\ell(z)}h_{w_0z,a}\underline{H}_z.
\end{displaymath}
Plugging this into \eqref{eq-prop-s2.7-1-2}, gives
\begin{equation}\label{eq-prop-s2.7-1-3}
\sum_{a,z\in W}(-1)^{\ell(w_0a)-\ell(z)}h_{w_0z,a} h_{yw_0,a}(v^{-1})
\underline{H}_{xw_0}\underline{H}_z H_{w_0}.
\end{equation}
Using \eqref{eq-strconst}, we rewrite \eqref{eq-prop-s2.7-1-3} as
\begin{displaymath}
\sum_{a,z,r\in W}(-1)^{\ell(w_0a)-\ell(z)}h_{w_0z,a} h_{yw_0,a}(v^{-1})
{{}{}\mathbf{h}_{xw_0,z}^r}\underline{H}_rH_{w_0}.
\end{displaymath}
Here we see that the coefficient at $\underline{H}_rH_{w_0}$
equals
\begin{displaymath}
\sum_{a,z\in W}(-1)^{\ell(w_0a)-\ell(z)}h_{w_0z,a} h_{yw_0,a}(v^{-1})
{{}{}\mathbf{h}_{xw_0,z}^r}.
\end{displaymath}
Renaming $r=ww_0$, plugging this back into \eqref{eq-prop-s2.7-1-1} 
and using \eqref{eq-virk} implies the claim.
\end{proof}

\subsection{$\mathfrak{sl}_2$-example}\label{s2.8}

Consider the case $W=S_2=\{e,s\}$, that is
$\mathfrak{g}=\mathfrak{sl}_2$. In this case $\mathbf{H}$
has the standard basis $\{H_e,H_s\}$. We have:
\begin{displaymath}
\underline{H}_e=H_e\quad\text{ and }\quad
\underline{H}_s=H_s+vH_e,
\end{displaymath}
furthermore
\begin{displaymath}
\underline{\hat{H}}_e=H_e-vH_s\quad\text{ and }\quad
\underline{\hat{H}}_s=H_s.
\end{displaymath}
The structure constants are given by the following two
tables, describing {{}{}$\mathbf{h}_{x,y}^e$} and 
{{}{}$\mathbf{h}_{x,y}^s$,} respectively:
\begin{displaymath}
\begin{array}{c||c|c}
x\setminus y & e& s\\
\hline\hline
e& 1& 0\\
\hline
s& 0& 0\\
\end{array}\qquad
\begin{array}{c||c|c}
x\setminus y & e& s\\
\hline\hline
e& 0& 1\\
\hline
s& 1& v+v^{-1}\\
\end{array}
\end{displaymath}
The dual structure constants are given by the following two
tables, describing {{}{}$\hat{\mathbf{h}}_{x,y}^e$} and 
{{}{}$\hat{\mathbf{h}}_{x,y}^s$,} respectively:
\begin{displaymath}
\begin{array}{c||c|c}
x\setminus y & e& s\\
\hline\hline
e&1+v^2& -v\\
\hline
s& -v& 1\\
\end{array}\qquad
\begin{array}{c||c|c}
x\setminus y & e& s\\
\hline\hline
e& 0& 0\\
\hline
s& 0& v^{-1}\\
\end{array}
\end{displaymath}

\section{Cyclic submodules of the regular Hecke module}\label{s3}

\subsection{Cyclic submodules for standard basis}\label{s3.1}

Consider the left regular $\mathbf{H}$-module
${}_\mathbf{H}\mathbf{H}$. Since all elements of the
standard basis are invertible, we have 
$\mathbf{H}H_w=\mathbf{H}$, for each $w\in W$.
Similarly, $H_w\mathbf{H}=\mathbf{H}$, for each $w\in W$.

\subsection{Multiplication of KL and dual KL basis elements}\label{s3.2}

Multiplication of the elements of the KL basis with the elements of 
the dual KL basis has some interesting properties
(see, e.g. \cite[Lemma~12]{MM1}):
\begin{itemize}
\item For $x,y\in W$, we have $\underline{H}_x\underline{\hat{H}}_y\neq 0$
if and only if $y\geq_R x^{-1}$ if and only if $y^{-1}\geq_L x$.
\item For $x,y\in W$, we have $\underline{\hat{H}}_y\underline{H}_x\neq 0$
if and only if $y\geq_L x^{-1}$ if and only if $y^{-1}\geq_R x$.
\end{itemize}
Furthermore, we have (see, e.g. \cite[Lemma~13]{MM1}):
\begin{itemize}
\item For any $u\in\mathbf{H}$ and any $x,y\in W$, if
$\underline{\hat{H}}_y$ appears with a non-zero coefficient in the
expression of $u\underline{\hat{H}}_x$ in the dual KL basis,
then $y\leq_L x$ (moreover, for each $y\leq_L x$, there is a $u$
for which $\underline{\hat{H}}_y$ does appear with a non-zero coefficient in 
$u\underline{\hat{H}}_x$).
\item For any $u\in\mathbf{H}$ and any $x,y\in W$, if
$\underline{\hat{H}}_y$ appears with a non-zero coefficient in the
expression of $\underline{\hat{H}}_xu$ in the dual KL basis,
then $y\leq_R x$ (moreover, for each $y\leq_R x$, there is a $u$
for which $\underline{\hat{H}}_y$ does appear with a non-zero coefficient
in $\underline{\hat{H}}_xu$).
\end{itemize}

\subsection{Cyclic submodules generated by the elements of 
the KL basis}\label{s3.3}

For an element $w\in W$, consider the cyclic submodule $\mathbf{H}\underline{H}_w$
of ${}_\mathbf{H}\mathbf{H}$. As {{}{}the} $\underline{H}_w$ are no longer
invertible, in general (i.e. unless $w=e$), we have
$\mathbf{H}\underline{H}_w\subsetneq \mathbf{H}$.
We would like to ask the following question.

\begin{question}\label{prob-inclKL}
{\hspace{1mm}}

\begin{enumerate}[$($a$)$]
\item\label{prob-inclKL.1}
For which $x,y\in W$, do we have 
$\mathbf{H}\underline{H}_x\subset \mathbf{H}\underline{H}_y$?
\item\label{prob-inclKL.2}
For which $x,y\in W$, do we have 
$\mathbf{H}\underline{H}_x=\mathbf{H}\underline{H}_y$?
\item\label{prob-inclKL.3}
For which $x,y\in W$, do we have 
$\mathbf{H}\underline{H}_x\cong \mathbf{H}\underline{H}_y$
(as $\mathbf{H}$-modules)?
\end{enumerate}
\end{question}

We start with the following observation.

\begin{proposition}\label{prop-s3.3-1}
For $x,y\in W$, the inclusion 
$\mathbf{H}\underline{H}_x\subset \mathbf{H}\underline{H}_y$
implies $x\geq_L y$.
\end{proposition}

\begin{proof}
The inclusion 
$\mathbf{H}\underline{H}_x\subset \mathbf{H}\underline{H}_y$
implies, in particular, that 
$\underline{H}_x\in\mathbf{H}\underline{H}_y$.
Now $x\geq_L y$ follows directly from the definition
of $\geq_L$.
\end{proof}

Unfortunately, it is very surprising to realize that $x\geq_L y$ 
does not imply 
$\mathbf{H}\underline{H}_x\subset \mathbf{H}\underline{H}_y$, in general.
Here is an explicit example.

\begin{example}\label{ex-s3.3-2}
{\em 
Let $W=S_3=\{e,s,t,st,ts,w_0=sts=tst\}$. In this case we have the
following KL basis elements:
\begin{displaymath}
\begin{array}{rcl}
\underline{H}_e&=&H_e\\
\underline{H}_s&=&H_s+vH_e\\
\underline{H}_t&=&H_t+vH_e\\
\underline{H}_{st}&=&H_{st}+vH_s+vH_t+v^2H_e\\
\underline{H}_{ts}&=&H_{ts}+vH_s+vH_t+v^2H_e\\
\underline{H}_{w_0}&=&H_{w_0}+vH_{st}+vH_{ts}+v^2H_s+v^2H_t+v^3H_e
\end{array}
\end{displaymath}
Here is the multiplication table in the KL basis:

\resizebox{\textwidth}{!}{
$
\begin{array}{c||c|c|c|c|c|c}
\cdot&\underline{H}_e&\underline{H}_s&\underline{H}_t&
\underline{H}_{st}&\underline{H}_{ts}&\underline{H}_{w_0}\\
\hline\hline
\underline{H}_e&\underline{H}_e&\underline{H}_s&\underline{H}_t&
\underline{H}_{st}&\underline{H}_{ts}&\underline{H}_{w_0}\\
\hline
\underline{H}_s&\underline{H}_s&(v+v^{-1})\underline{H}_s&
\underline{H}_{st}&
(v+v^{-1})\underline{H}_{st}&\underline{H}_{s}+
\underline{H}_{w_0}&(v+v^{-1})\underline{H}_{w_0}\\
\hline
\underline{H}_t&\underline{H}_t&\underline{H}_{ts}&
(v+v^{-1})\underline{H}_{t}&
\underline{H}_{t}+
\underline{H}_{w_0}&(v+v^{-1})\underline{H}_{ts}&(v+v^{-1})\underline{H}_{w_0}\\
\hline
\underline{H}_{st}&\underline{H}_{st}&\underline{H}_{s}+
\underline{H}_{w_0}&
(v+v^{-1})\underline{H}_{st}&
\underline{H}_{st}+
(v+v^{-1})\underline{H}_{w_0}&
(v+v^{-1})(\underline{H}_{s}+
\underline{H}_{w_0})&(v^2+2+v^{-2})\underline{H}_{w_0}\\
\hline
\underline{H}_{ts}&\underline{H}_{ts}&
(v+v^{-1})\underline{H}_{ts}&
\underline{H}_{t}+
\underline{H}_{w_0}&
(v+v^{-1})(\underline{H}_{t}+
\underline{H}_{w_0})&\underline{H}_{ts}+
(v+v^{-1})\underline{H}_{w_0}&(v^2+2+v^{-2})\underline{H}_{w_0}\\
\hline
\underline{H}_{w_0}&\underline{H}_{w_0}&
(v+v^{-1})\underline{H}_{w_0}&
(v+v^{-1})\underline{H}_{w_0}&
(v^2+2+v^{-2})\underline{H}_{w_0}&
(v^2+2+v^{-2})\underline{H}_{w_0}&
(v^3+2v+2v^{-1}+v^{-3})\underline{H}_{w_0}\\
\hline
\end{array}
$
}

We have four left cells: 
\begin{displaymath}
\mathcal{L}_e:=\{e\},\quad 
\mathcal{L}_s:=\{s,ts\},\quad 
\mathcal{L}_t:=\{t,st\},\quad 
\mathcal{L}_{w_0}:=\{w_0\}. 
\end{displaymath}

Looking  at the second column of the multiplication table, 
we see that $\mathbf{H}\underline{H}_s$ is a free $\mathbb{A}$-module
with basis $\underline{H}_s$, $\underline{H}_{ts}$ and $\underline{H}_{w_0}$.
In particular, $\mathbf{H}\underline{H}_{ts}\subset \mathbf{H}\underline{H}_s$.
At the same time, looking  at the forth column of the multiplication table, 
we see that $\mathbf{H}\underline{H}_{ts}$ is a free $\mathbb{A}$-module
with basis $\underline{H}_{ts}$, $\underline{H}_{s}+\underline{H}_{w_0}$
and $(v+v^{-1})\underline{H}_{w_0}$. In particular, 
$\mathbf{H}\underline{H}_{s}\not\subset \mathbf{H}\underline{H}_{ts}$.

Here we note that one could force the inclusion 
$\mathbf{H}\underline{H}_{s}\subset \mathbf{H}\underline{H}_{ts}$
by extending the scalars such that $(v+v^{-1})$ becomes invertible.\hfill
\rule{1.2ex}{1.2ex} 
}
\end{example}

For $w\in W$, we denote by $\mathbf{LM}_w$ the 
$\mathbb{A}$-span of all $\underline{H}_u$, with $u\geq_L w$.
Then $\mathbf{LM}_w$ is a left $\mathbf{H}$-submodule
of ${}_\mathbf{H}\mathbf{H}$ and
$\mathbf{H}\underline{H}_w\subset \mathbf{LM}_w$.
Note that $\mathbf{LM}_x=\mathbf{LM}_y$ if and only if
$x\sim_L y$, by definition. We note that Example~\ref{ex-s3.3-2}
implies that $\mathbf{LM}_w\neq \mathbf{H}\underline{H}_w$,
in general. Therefore the following question looks natural:

\begin{question}\label{prob-inclKL-2}
For which $w\in W$, do we have 
$\mathbf{H}\underline{H}_w=\mathbf{LM}_w$?
\end{question}

Now we give another example which shows that the 
above questions are even more difficult than 
what is indicated by Example~\ref{ex-s3.3-2}.

\begin{example}\label{ex-s3.3-3}
{\em 
Let $W=S_4$. From \cite[Figure~1]{RT}
or \cite[Figure~6.10]{BB}, here is the Hasse
diagram for the restriction of $\leq_L$ to involutions in
$S_4$ (note that here involutions are identified with the
standard Young tableaux via $\mathbf{RS}$):

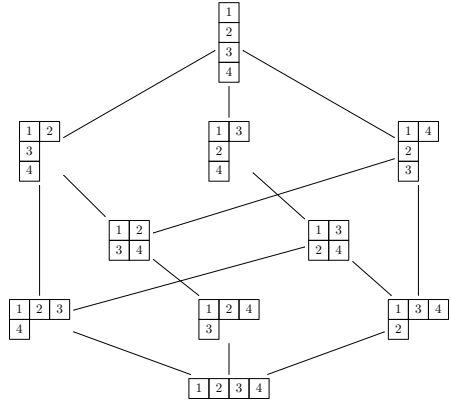
\begin{center}
\resizebox{6cm}{!}{
$
\xymatrix{
&&
\ytableausetup{centertableaux}
{\begin{ytableau}
1\\
2\\
3\\
4
\end{ytableau}}
\ar@{-}[dll]\ar@{-}[d]\ar@{-}[drr]
&&\\
{\begin{ytableau}
1&2\\
3\\
4
\end{ytableau}}
\ar@{-}[dd]\ar@{-}[dr]
&&
{\begin{ytableau}
1&3\\
2\\
4
\end{ytableau}}
\ar@{-}[dr]
&&
{\begin{ytableau}
1&4\\
2\\
3
\end{ytableau}}
\ar@{-}[dd]\ar@{-}[dlll]
\\
&
{\begin{ytableau}
1&2\\
3&4
\end{ytableau}}
\ar@{-}[dr]
&&
{\begin{ytableau}
1&3
\ar@{-}[dr]\ar@{-}[dlll]
\\
2&4
\end{ytableau}}&\\
{\begin{ytableau}
1&2&3\\
4
\end{ytableau}}
\ar@{-}[drr]
&&
{\begin{ytableau}
1&2&4\\
3
\end{ytableau}}
\ar@{-}[d]&&
{\begin{ytableau}
1&3&4\\
2
\end{ytableau}}
\ar@{-}[dll]\\
&&
{\begin{ytableau}
1&2&3&4
\end{ytableau}}
&&
}
$
}
\end{center}

Here is the list of the sizes of left cells for each involution
(these are given by the Hook Formula):
\begin{displaymath}
\begin{array}{c||c|c|c|c|c|c|c|c|c|c}
w&
\resizebox{5mm}{!}{\begin{ytableau}1&2&3&4 \end{ytableau}}&
\resizebox{4mm}{!}{\begin{ytableau}1&2&3\\4 \end{ytableau}}&
\resizebox{4mm}{!}{\begin{ytableau}1&2&4\\3 \end{ytableau}}&
\resizebox{4mm}{!}{\begin{ytableau}1&3&4\\2 \end{ytableau}}&
\resizebox{3mm}{!}{\begin{ytableau}1&2\\3&4 \end{ytableau}}&
\resizebox{3mm}{!}{\begin{ytableau}1&3\\2&4 \end{ytableau}}&
\resizebox{2mm}{!}{\begin{ytableau}1&2\\3\\4 \end{ytableau}}&
\resizebox{2mm}{!}{\begin{ytableau}1&3\\2\\4 \end{ytableau}}&
\resizebox{2mm}{!}{\begin{ytableau}1&4\\2\\3 \end{ytableau}}&
\resizebox{1mm}{!}{\begin{ytableau}1\\2\\3\\4 \end{ytableau}}\\
\hline\hline
|\mathcal{L}_w|&1&3&3&3&2&2&3&3&3&1
\end{array}
\end{displaymath}

We see that the left coideal generated by the involution 
\resizebox{5mm}{!}{$\begin{ytableau}1&2\\3&4 \end{ytableau}$}
contains, in addition to itself, the involutions
\resizebox{5mm}{!}{$\begin{ytableau}1&2\\3\\4 \end{ytableau}$},
\resizebox{5mm}{!}{$\begin{ytableau}1&4\\2\\3 \end{ytableau}$}
and \resizebox{3mm}{!}{$\begin{ytableau}1\\2\\3\\4 \end{ytableau}$}
and hence $9$ elements of $W$ in total.

At the same time, the left coideal generated by the involution 
\resizebox{5mm}{!}{$\begin{ytableau}1&3\\2&4 \end{ytableau}$}
contains, in addition to itself, only the involutions
\resizebox{5mm}{!}{$\begin{ytableau}1&3\\3\\4 \end{ytableau}$},
and \resizebox{3mm}{!}{$\begin{ytableau}1\\2\\3\\4 \end{ytableau}$}
and hence only $6$ elements of $W$ in total.

Now let $w$ be the unique element which belongs to the left cell
of the involution 
\resizebox{5mm}{!}{$\begin{ytableau}1&2\\3&4 \end{ytableau}$}
and the right cell of the involution 
\resizebox{5mm}{!}{$\begin{ytableau}1&3\\2&4 \end{ytableau}$}.
On the one hand, there are $9$ elements $u\in S_4$
such that $u\geq_L w$. On the other hand, there are only
$6$ elements $v\in W$ such that 
$\underline{\hat{H}}_v\underline{H}_w\neq 0$.
Therefore the $\mathbb{A}$-rank of 
$\mathbf{H}\underline{H}_w$ is at most $6$
while the $\mathbb{A}$-rank of $\mathbf{LM}_w$ is $9$.
Hence $\mathbf{H}\underline{H}_w\neq \mathbf{LM}_w$
and this cannot be rectified by any extension of scalars
that preserves the $\mathbb{A}$-rank.\hfill
\rule{1.2ex}{1.2ex} 
}
\end{example}

Later on, in Theorem~\ref{thm-s3.3-4}, we will show that 
Question~\ref{prob-inclKL-2} has {{}{}a} positive answer for 
longest elements of parabolic subgroups of $W$.
The longest elements of parabolic subgroups
$W$ enjoy a number of remarkable properties, 
see e.g. \cite[Proposition~4.1]{CM} or \cite[Theorem~7.3]{MT}.
From the point of view of the KL-combinatorics, such elements can be
distinguished, for example, by the following property:

\begin{proposition}\label{prop-s3.3-5}
For $w\in W$, we have $\underline{H}_w^2= a\underline{H}_w$,
for some $a\in \mathbb{A}$, if and only if 
$w$ is the longest element $w_0'$ in some parabolic subgroup $W'$ of $W$.
\end{proposition}

{{}{}We suspect that this claim 
can already be in the literature, but, unfortunately, 
we did not manage to find a proper reference and hence give a complete proof.}
To prove Proposition~\ref{prop-s3.3-5}, we will need {{}{}a few} 
auxiliary lemmata.

{{}{}
\begin{lemma}\label{lem-inrev}
For any $a,b\in W$ such that $\ell(ab)=\ell(a)+\ell(b)$,
we have
\begin{equation}\label{eq-aaa-n1}
\underline{H}_{a}\underline{H}_{b}=
\underline{H}_{ab}+u,
\end{equation}
where $u$ is an $\mathbb{A}$-linear combination of
$\underline{H}_{c}$, where $c\prec ab$. 
\end{lemma}

\begin{proof}
Recall that, for any $a\in W$ and 
any reduced expression $a=s_1s_2\dots s_k$, we have
\begin{displaymath}
\underline{H}_{s_1}\underline{H}_{s_2}
\dots \underline{H}_{s_k}=
\underline{H}_{a}+u,
\end{displaymath}
where $u$ is an $\mathbb{A}$-linear combination of
$\underline{H}_{b}$, with $b\prec a$, 
see e.~g. the proof of \cite[Theorem~2.1]{So}.  
The claim of the lemma follows by concatenating reduced
expressions for $a$ and $b$.
\end{proof}
}

\begin{lemma}\label{lem-s3.3-51}
Let $W'$ be a parabolic subgroup of $W$ and $w\in W$.
Let $u$ be the longest representative in the coset $W'w$
and $w'_0$ the longest element of $W'$. Then
${{}{}\mathbf{h}_{w'_0,w}^u}\neq 0$.
\end{lemma}

\begin{proof}
Recall that, for $x\in W$, we have $\theta_xT_{w_0}\cong T_{w_0x}$.
Therefore, in order to prove our lemma, we just need to show that 
$\theta_w(\theta_{w'_0}T_{w_0})$ contains $T_{w_0u}$, up to a shift 
of grading. We also note that each $T_{w_0x}$ has $L_{w_0x}$ as 
a simple subquotient and all other simple subquotients $L_y$ of 
$T_{w_0x}$ satisfy $w_0x\preceq y$. Therefore, we just need to show that
$\theta_w(\theta_{w'_0}T_{w_0})$ contains $L_{w_0u}$ as 
a simple subquotient and all other simple subquotients $L_y$ of 
$\theta_w(\theta_{w'_0}T_{w_0})$ satisfy {{}{}$w_0u\preceq y$}. 

Let $v\in W'$ be such that $vw=u$, in particular, $\ell(v)+\ell(w)=\ell(u)$.
Let also $z\in W'$ be such that $zv=w'_0$, in particular,
$\ell(z)+\ell(v)=\ell(w'_0)$. From the latter equality, we deduce that
$\underline{H}_z\underline{H}_v=\underline{H}_{w'_0}+r$,
where $r$ is an $\mathbb{A}$-linear combination of
$\underline{H}_{c}$, with $c\prec w'_0$, see \eqref{eq-aaa-n1}.
Consider now $\underline{H}_z\underline{H}_v\underline{H}_w$. Here
$\underline{H}_v\underline{H}_w=\underline{H}_{u}+\tilde{r}$,
where $\tilde{r}$ is an $\mathbb{A}$-linear combination of
$\underline{H}_{c}$, with $c\prec u$, again by \eqref{eq-aaa-n1}.

Let $Q=\{x\in W\,:\, x\preceq u\}$. For any simple reflection $s\in W'$, 
we have $\ell(su)<\ell(u)$ due to our choice of $u$ as a longest
representative in $W'w$. But this means that, for any $x\in Q$,
the product $\underline{H}_s\underline{H}_x$ is an 
$\mathbb{A}$-linear combination of $\underline{H}_{y}$, where $y\in Q$.
For any reduced expression $z=s_1s_2\dots s_k$, we have that 
$\underline{H}_{s_1}\underline{H}_{s_2}\dots
\underline{H}_{s_k}$ equals $\underline{H}_z+\tilde{\tilde{r}}$, where 
where $\tilde{\tilde{r}}$ is an $\mathbb{A}$-linear combination of
$\underline{H}_{c}$, with $c\prec z$, by \eqref{eq-aaa-n1}.
Consequently, $\underline{H}_z\underline{H}_v\underline{H}_w$
is an $\mathbb{A}$-linear combination of
$\underline{H}_{c}$, with $c\prec u$, and hence 
$\underline{H}_{w'_0}\underline{H}_w$ is such an 
$\mathbb{A}$-linear combination as well.

It remains to prove that 
$\theta_w(\theta_{w'_0}T_{w_0})$ contains $L_{w_0u}$ as 
a simple subquotient. The module $\theta_{w'_0}T_{w_0}$
contains $\Delta_{w_0w'_0}$ as a submodule and thus 
has all $\Delta_{x}$, where $x$ satisfies $w_0\succeq x\succeq w_0w'_0$, as
submodules as well. We claim that 
$\theta_w\Delta_{w_0v}$ contains $L_{w_0u}$ as 
a simple subquotient. Indeed, $\Delta_{w_0v}=\top_{w_0vw_0}T_{w_0}$,
where $\top_{w_0vw_0}$ is the twisting functor for
$w_0vw_0$, see \cite{AS}, and hence 
\begin{displaymath}
\theta_w\Delta_{w_0v}=
\theta_w(\top_{w_0vw_0}T_{w_0}).
\end{displaymath}
As projective and twisting functors commute, see \cite[Section~3]{AS}, we have 
\begin{displaymath}
\theta_w\Delta_{w_0v}=
\top_{w_0vw_0}(\theta_wT_{w_0})=
\top_{w_0vw_0}T_{w_0w}.
\end{displaymath}
From the definition of a tilting module, we have 
$T_{w_0w}\tto\nabla_{w_0w}$. Since $\top_{w_0vw_0}$
is right exact, it follows that 
$\top_{w_0vw_0}T_{w_0w}\tto\top_{w_0vw_0}\nabla_{w_0w}$.
Finally, $\top_{w_0vw_0}\nabla_{w_0w}=\nabla_{w_0u}$
by \cite[Theorem~2.3]{AS}. As $\nabla_{w_0u}$, clearly, has 
$L_{w_0u}$ as a subquotient, the claim follows.
This completes the proof.
\end{proof}

The following lemma gives yet another description of the
longest elements in parabolic subgroups.

\begin{lemma}\label{lem-afunction}
If $w\in W$ is such that $\mathbf{a}(w)=\ell(w)$, then
$w$ is the longest element in some parabolic subgroup of $W$.
\end{lemma}

\begin{proof}
For any $x\in W$, we have $[\Delta_e:L_x\langle -\ell(x)\rangle]=1$.
Therefore $\mathbf{a}(w)=\ell(w)$ implies 
$[\Delta_e:L_w\langle -\mathbf{a}(w)\rangle]=1$, which implies
that $w$ is a Duflo involution, by the definition of the latter.

The module $\theta_w L_w$ is concentrated between the degrees
$\pm\mathbf{a}(w)=\pm \ell(w)$ with the two extreme degrees just
consisting of a copy of $L_w$. Let $W'$ be the smallest parabolic
subgroup of $W$ containing $w$ and let $w'_0$ be its longest element.
Let $w=s_1s_2\dots s_k$ be a reduced expression. For each $i$,
the module $\theta_{s_i}L_w$ is either zero, if $ws_i\succ w$,
or, otherwise  concentrated
between the degrees $\pm 1$ with the two extreme degrees containing
just a copy of $L_w$. As $\theta_w$ is a summand of 
$\theta_{s_k}\dots \theta_{s_2}\theta_{s_1}$, it follows that 
$\theta_{s_i}L_w\neq 0$, for all $i$. This means that 
$ws_i\prec w$, for each $i$. But then $ww'_0\prec w$, implying
$w=w'_0$.
\end{proof}

Now we can prove Proposition~\ref{prop-s3.3-5}.

\begin{proof}[Proof of Proposition~\ref{prop-s3.3-5}]
If $w=w'_0$, then
\begin{displaymath}
\underline{H}_{w'_0}=\sum_{x\in W'}v^{\ell(w_0)-\ell(x)}H_x,
\end{displaymath}
see \cite[Proposition~2.7]{So}. The equality
\begin{displaymath}
\underline{H}^2_{w'_0}=
\big(\sum_{x\in W'}v^{\ell(w'_0)-2\ell(x)}\big)
\underline{H}_{w'_0}
\end{displaymath}
follows, {{}{}for example, from
\cite[Formula~(2.9)]{Wi} or} \cite[Proposition~4.1]{CM}.

To prove the converse statement, we will prove the following
claim: for $w\in W$, let $\hat{W}$ be the minimal parabolic
subgroup of $W$ which contains $w$. Then, for some positive
integer $k$, the element $\underline{H}_{\hat{w}_0}$, where
$\hat{w}_0$ is the longest element of $\hat{W}$, appears with a non-zero
coefficient when expressing $\underline{H}_w^k$ 
with respect to the KL basis.

To prove this claim, we do induction with 
respect to the rank of $W$ as well as reverse induction 
with respect to the length of $w$. If $W$ has rank $1$,
then all elements of $W$ are longest elements of some
parabolic subgroups, so there is nothing to prove.
Also, the longest element of $W$ is, obviously, the
longest element of a parabolic subgroup, so the bases
of both induction work just fine. Let us do the 
induction step.

Assume that we have an element $w\in W$ of maximal
possible length such that, for all $k\in\mathbb{Z}_{\geq 0}$,
the element $\underline{H}_{w_0}$ does not appear with a non-zero
coefficient when expressing $\underline{H}^k_{w}$ with 
respect to the KL basis. By our first induction
(with respect to the rank of $W$), we know
that $w$ does not belong to any proper parabolic subgroup
of $W$. Define $Q$ as the set of all $x\in W$ for which
$\underline{H}_{x}$ does appear with a non-zero
coefficient when expressing $\underline{H}^k_{w}$ with 
respect to the KL basis, for some positive integer $k$.
Note that, directly from the definition we get that
the $\mathbb{A}$-span of $\underline{H}_{x}$,
where $x\in Q$, is closed under multiplication.

Assume that some $x\in Q$ belongs to some proper parabolic
subgroup $W'$ of $W$. Then, by our first induction, 
$\underline{H}_{w'_0}$ does appear with a non-zero
coefficient when expressing $\underline{H}^k_{x}$ with 
respect to the KL basis, for some positive integer $k$.
Hence $w'_0\in Q$. But then, from Lemma~\ref{lem-s3.3-51},
we obtain that $Q$ contains the longest 
representative $u$ in $W'w$. Obviously, $u\succeq w$, in 
particular, $u$ does not belong to any proper
parabolic subgroup of $W$. Hence, by our assumption
that $w$ has maximal possible length, it follows
that $u=w$. In particular, $\ell(w)\geq \ell(w'_0)$
and, if $w\neq w'_0$, we do have 
\begin{equation}\label{eq-nn-326}
\ell(w)> \ell(w'_0). 
\end{equation}
This implies that $\ell(w)$ is maximal among
all $\ell(x)$, where $x\in Q$.

Recall that, for $w\in W$ and a simple reflection $s$,
the module $\theta_s L_w$ is non-zero if and only if $ws\prec w$,
moreover, in the latter case, $\theta_s L_w$ lives in degrees
$0,\pm 1$ with simple top $L_w$ in degree $-1$ and simple
socle $L_w$ in degree $1$. By induction on the length, it follows
that, for and $a\in W$, the module $\theta_a L_{w_0}$
is concentrated between the degrees $-\ell(a)$ and $\ell(a)$,
moreover, both extreme degrees consist just of a copy of
$L_{w_0}$. Consequently, applying $\underline{H}^k_{w}$ to
$L_{w_0}$ gives a module concentrated between the degrees
$-k\ell(w)$ and $k\ell(w)$, moreover, both extreme degrees 
consist just of a copy of $L_{w_0}$.

On the other hand, the element $w$ is certainly not the
longest element of any parabolic subgroup, hence
$\mathbf{a}(w)<\ell(w)$. Also, for any $x\in Q$,
we have $\mathbf{a}(x)\leq \ell(x)\leq  \ell(w)$ as we know that
$w$ is a longest element in $Q$. Note that, for any $x$,
at least one of these inequalities is strict.
Indeed, if $\mathbf{a}(x)=\ell(x)$, then $x=w'_0$,
for some proper parabolic subgroup $W'$ and we have 
$\ell(x)<\ell(w)$ from \eqref{eq-nn-326}.
We have
\begin{displaymath}
\underline{H}^2_{w}=\sum_{x\in Q}a_x \underline{H}_{x},
\end{displaymath}
where $a_x\in\mathbb{A}$ is a polynomial whose all non-zero
monomials have degrees between $-\mathbf{a}(x)$ and $\mathbf{a}(x)$.
Also, $\ell(x)\leq\ell(w)$. Therefore the module
\begin{displaymath}
\sum_{x\in Q}a_x \underline{H}_{x}L_{w_0}
\end{displaymath}
must be zero in the degrees $\pm 2\ell(w)$, a contradiction.
The claim of the proposition follows.
\end{proof}

The following observation explains, in greater generality,
some parts of Example~\ref{ex-s3.3-2}:

\begin{theorem}\label{thm-s3.3-4}
Let $W'$ be a parabolic subgroup of $W$
and $w'_0$ the longest element of $W'$.
Then $\mathbf{H}\underline{H}_{w'_0}=\mathbf{LM}_{w'_0}$.
\end{theorem}

\begin{proof}
To start with, let us recall that, 
for $x,y\in W$ and a simple reflection $s$, 
the inequality $x\geq_L y$ yields that $ys\prec y$ implies 
$xs\prec x$, see e.~g. \cite[Proposition~6.2.7]{BB}. 
Therefore, for any $w\geq_L w'_0$ and any simple
reflection $s$ such that $w'_0s\prec w'_0$, we have
$ws\prec w$. This means that we can write 
\begin{equation}\label{eq-prop-s3.3-4-1}
w=xw'_0, \text{ where } x\in W, 
\text{ such that } \ell(w)=\ell(x)+\ell(w'_0).
\end{equation}

{{}{}Applying Lemma~\ref{lem-inrev} to \eqref{eq-prop-s3.3-4-1}},
we obtain that $\underline{H}_{x}\underline{H}_{w'_0}$
equals $\underline{H}_{w}$ plus an $\mathbb{A}$-linear combination of
$\underline{H}_{c}$, where $w'_0\preceq c\prec xw'_0$. By induction with 
respect to $\prec$ we thus obtain that 
$\mathbf{H}\underline{H}_{w'_0}$ contains
$\underline{H}_{w}$. This implies the claim of the proposition.
\end{proof}

If we carefully look at the proof of Theorem~\ref{thm-s3.3-4}
and single out the properties of $w'_0$ that 
we really used, we see that the same argument
proves the following:

\begin{corollary}\label{cor-s3.3-45}
Let $w\in W$ be such that $\{x\in W\,:\,x\geq_L w\}$
coincides with 
\begin{displaymath}
\{y\in W \,:\, y=aw\text{ for some }
a\in W\text{ such that }\ell(aw)=\ell(a)+\ell(w)\}. 
\end{displaymath}
Then {{}{}$\mathbf{H}\underline{H}_{w}=\mathbf{LM}_{w}$}.
\end{corollary}

\begin{proof}
Mutatis mutandis the proof of Theorem~\ref{thm-s3.3-4}.
\end{proof}

\begin{remark}\label{rem-13}
{\em
The only examples of $w$ satisfying the assumptions of 
Corollary~\ref{cor-s3.3-45} that we know are the longest 
elements of parabolic subgroups. It would be interesting to 
know whether the condition of Corollary~\ref{cor-s3.3-45}
gives yet another characterization of these longest elements.

For instance, in $S_4$, there are exactly two KL left cells
that do not contain any longest elements of any 
parabolic subgroup. If we denote the simple reflections by $s_1,s_2,s_3$
with $s_1s_3=s_3s_1$, then the corresponding two Duflo involutions
are $s_2s_1s_3s_2$ and $s_2w_0$. For 
$d\in\{s_2s_1s_3s_2,s_2w_0\}$, the 
left coideal $\{x\in W\,:\, x\geq_L d\}$ has the following
properties: the set of the elements of minimal length in
this coideal is not a singleton and the Duflo involution
itself is not a minimal length element in this coideal.
In particular, for $d=s_2w_0$, the corresponding coideal 
{{}{}consists} of
$w_0$, $d$ and the elements $s_1d$ and $s_3d$, both of length
$\ell(d)-1$.

From the small rank computations that we have, it is very tempting to
expect that, for a fixed $w\in W$, 
the condition 
\begin{displaymath}
\underline{H}_x\in \mathbf{H}\underline{H}_w,
\text{ for all } x\geq_L w,
\end{displaymath}
implies that $w$ must be the longest element
of some parabolic subgroup. However, we feel that 
{{}{}the} small rank computations
that we have so far are not really enough to formulate this 
expectation as a conjecture.
\hfill
\rule{1.2ex}{1.2ex} 
}
\end{remark}

\subsection{Cyclic submodules generated by the elements of 
the dual KL basis}\label{s3.4}

For $w\in W$, consider the cyclic submodule $\mathbf{H}\underline{\hat{H}}_w$
of ${}_\mathbf{H}\mathbf{H}$. As $\underline{\hat{H}}_w$ is not
invertible, we have
$\mathbf{H}\underline{\hat{H}}_w\subsetneq \mathbf{H}$.
Similarly to the case of the KL basis elements, it is 
natural to ask the following question.

\begin{question}\label{prob-incldKL}
{\hspace{1mm}}

\begin{enumerate}[$($a$)$]
\item\label{prob-incldKL.1}
For which $x,y\in W$, do we have 
$\mathbf{H}\underline{\hat{H}}_x\subset \mathbf{H}\underline{\hat{H}}_y$?
\item\label{prob-incldKL.2}
For which $x,y\in W$, do we have 
$\mathbf{H}\underline{\hat{H}}_x=\mathbf{H}\underline{\hat{H}}_y$?
\item\label{prob-incldKL.3}
For which $x,y\in W$, do we have 
$\mathbf{H}\underline{\hat{H}}_x\cong \mathbf{H}\underline{\hat{H}}_y$
(as $\mathbf{H}$-modules)?
\end{enumerate}
\end{question}

We start with the following observation.

\begin{proposition}\label{prop-s3.4-1}
For $x,y\in W$, the inclusion 
$\mathbf{H}\underline{\hat{H}}_x\subset \mathbf{H}\underline{\hat{H}}_y$
implies $x\leq_L y$.
\end{proposition}

\begin{proof}
The inclusion 
$\mathbf{H}\underline{\hat{H}}_x\subset \mathbf{H}\underline{\hat{H}}_y$
implies, in particular, that 
$\underline{\hat{H}}_x\in\mathbf{H}\underline{\hat{H}}_y$.
Equivalently, there is $z\in W$ such that 
$(\underline{H}_x,\underline{H}_z\underline{\hat{H}}_y)\neq 0$.
Moving $\underline{H}_z$ to the {{}{}left} hand side, this is equivalent
to ${{}{}\mathbf{h}_{z^{-1},x}^y}\neq 0$. 
Now $x\leq_L y$ follows directly from the definition
of the left order.
\end{proof}

Similarly to the case of the ordinary KL basis,  $x\leq_L y$ 
does not imply {{}{}
$\mathbf{H}\hat{\underline{H}}_x\subset 
\mathbf{H}\hat{\underline{H}}_y$}, in general.
Here is an explicit example.

\begin{example}\label{ex-s3.4-2}
{\em 
Let $W=S_3=\{e,s,t,st,ts,w_0=sts=tst\}$. In this case we have the
following dual KL basis elements:
\begin{displaymath}
\begin{array}{lcl}
\underline{\hat{H}}_e&=&H_e-vH_{s}-vH_{t}+v^2H_{st}+v^2H_{ts}-v^3H_{w_0}\\
\underline{\hat{H}}_s&=&H_s-vH_{st}-vH_{ts}+v^2H_{w_0}\\
\underline{\hat{H}}_t&=&H_t-vH_{st}-vH_{ts}+v^2H_{w_0}\\
\underline{\hat{H}}_{st}&=&H_{st}-{{}{}vH_{w_0}}\\
\underline{\hat{H}}_{ts}&=&H_{ts}-vH_{w_0}\\
\underline{\hat{H}}_{w_0}&=&H_{w_0}
\end{array}
\end{displaymath}
Here is the multiplication table of the elements of the KL basis
with the elements of the dual KL basis, with the outcome expressed
in the dual KL basis:

\resizebox{\textwidth}{!}{
$
\begin{array}{c||c|c|c|c|c|c}
\cdot&\underline{\hat{H}}_e&\underline{\hat{H}}_s&\underline{\hat{H}}_t&
\underline{\hat{H}}_{st}&\underline{\hat{H}}_{ts}&\underline{\hat{H}}_{w_0}\\
\hline\hline
\underline{H}_e&\underline{\hat{H}}_e&\underline{\hat{H}}_s&\underline{\hat{H}}_t&
\underline{\hat{H}}_{st}&\underline{\hat{H}}_{ts}&\underline{\hat{H}}_{w_0}\\
\hline
\underline{H}_s&0&(v+v^{-1})\underline{\hat{H}}_s+
\underline{\hat{H}}_e+\underline{\hat{H}}_{ts}&
0&
(v+v^{-1})\underline{\hat{H}}_{st}+\underline{\hat{H}}_t&0
&(v+v^{-1})\underline{\hat{H}}_{w_0}+\underline{\hat{H}}_{ts}\\
\hline
\underline{H}_t&0&0&
(v+v^{-1})\underline{\hat{H}}_{t}+\underline{\hat{H}}_{e}+
\underline{\hat{H}}_{st}&0&
(v+v^{-1})\underline{\hat{H}}_{ts}+\underline{\hat{H}}_{s}&
(v+v^{-1})\underline{H}_{w_0}+\underline{\hat{H}}_{st}\\
\hline
\underline{H}_{st}&0&0&(v+v^{-1})\underline{\hat{H}}_{st}+\underline{\hat{H}}_t&0&
(v+v^{-1})\underline{\hat{H}}_s+
\underline{\hat{H}}_e+\underline{\hat{H}}_{ts}&X_1\\
\hline
\underline{H}_{ts}&0&(v+v^{-1})\underline{\hat{H}}_{ts}
+\underline{\hat{H}}_{s}&0&(v+v^{-1})\underline{\hat{H}}_{t}+\underline{\hat{H}}_{e}+
\underline{\hat{H}}_{st}&0&X_2\\
\hline
\underline{H}_{w_0}&0&0&0&0&0&X_3\\
\hline
\end{array}
$
}

Here 

\resizebox{\textwidth}{!}{
$
\begin{array}{rcl}
X_1&=&(v^{-2}+2+v^2)\underline{\hat{H}}_{w_0}+(v+v^{-1})
(\underline{\hat{H}}_{st}+\underline{\hat{H}}_{ts})+\underline{\hat{H}}_{t},\\
X_2&=&(v^{-2}+2+v^2)\underline{\hat{H}}_{w_0}+(v+v^{-1})
(\underline{\hat{H}}_{st}+\underline{\hat{H}}_{ts})+\underline{\hat{H}}_{s},\\
X_3&=&(v^{-3}+2v^{-1}+2v+v^3)\underline{\hat{H}}_{w_0}+
(v^{-2}+2+v^2)(\underline{\hat{H}}_{st}+\underline{\hat{H}}_{ts})
+(v+v^{-1})(\underline{\hat{H}}_{s}+\underline{\hat{H}}_{t})
+\underline{\hat{H}}_{e}.
\end{array}
$
}

The penultimate column of this table gives 
that $\mathbf{H}\underline{\hat{H}}_{ts}$ is a free $\mathbb{A}$-module
with basis $\underline{\hat{H}}_{ts}$, $\underline{\hat{H}}_{s}$ 
and $\underline{\hat{H}}_{e}$.
In particular, 
$\mathbf{H}\underline{\hat{H}}_{s}\subset \mathbf{\hat{H}}\underline{H}_{ts}$.
At the same time, looking  at the second column of the table, 
we see that $\mathbf{H}\underline{\hat{H}}_{s}$ is a free $\mathbb{A}$-module
with basis $\underline{\hat{H}}_{s}$, $\underline{\hat{H}}_{ts}+\underline{\hat{H}}_{e}$
and $(v+v^{-1})\underline{\hat{H}}_{ts}$. In particular, 
$\mathbf{H}\underline{\hat{H}}_{ts}\not\subset \mathbf{H}\underline{\hat{H}}_{s}$.

Here we again note that one could force the inclusion 
$\mathbf{H}\underline{\hat{H}}_{ts}\subset \mathbf{H}\underline{\hat{H}}_{s}$
by extending the scalars such that $(v+v^{-1})$ becomes invertible.\hfill
\rule{1.2ex}{1.2ex} 
}
\end{example}

For $w\in W$, we denote by $\mathbf{LN}_w$ the 
$\mathbb{A}$-span of all $\underline{\hat{H}}_u$, with $u\leq_L w$.
Then $\mathbf{LN}_w$ is a left $\mathbf{H}$-submodule
of ${}_\mathbf{H}\mathbf{H}$ and
$\mathbf{H}\underline{\hat{H}}_w\subset \mathbf{LN}_w$.
Note that $\mathbf{LN}_x=\mathbf{LN}_y$ if and only if
$x\sim_L y$, by definition. The above Example~\ref{ex-s3.4-2}
implies that $\mathbf{LN}_w\neq \mathbf{H}\underline{\hat{H}}_w$,
in general. Therefore the following question looks natural:

\begin{question}\label{prob-incldKL-2}
For which $w\in W$, do we have 
$\mathbf{H}\underline{\hat{H}}_w=\mathbf{LN}_w$?
\end{question}

We can answer this question in the following special case:

\begin{proposition}\label{prop-s3.4-4}
Let $W'$ be a parabolic subgroup of $W$
and $w'_0$ the longest element of $W'$.
Then $\mathbf{H}\underline{\hat{H}}_{w_0w'_0}=\mathbf{LN}_{w_0w'_0}$.
\end{proposition}

\begin{proof}
For $w\in W$, let $\mathbf{RN}_{w}$ denote the 
$\mathbb{A}$-span of all $\underline{\hat{H}}_u$, with $u\leq_R w$.
Using $({}_-)^*$, it is enough to prove that 
$\underline{\hat{H}}_{w'_0w_0}\mathbf{H}=\mathbf{RN}_{w'_0w_0}$.

Let $\mathfrak{p}$ be the parabolic subalgebra of $\mathfrak{g}$
containing $\mathfrak{h}\oplus\mathfrak{n}_+$ corresponding to $W'$.
Consider the corresponding parabolic category $\mathcal{O}^{\mathfrak{p}}$,
defined in \cite{RC}. The principal block $\mathcal{O}^{\mathfrak{p}}_0$
of $\mathcal{O}^{\mathfrak{p}}$ is the Serre subcategory of 
$\mathcal{O}_0$ generated by those simple $L_x$, for which $x$ belongs to the
set of the shortest representatives in cosets $W'w$, where $w\in W$.
In particular, $w'_0w_0$ is the longest element in this indexing set.
So, $L_{w'_0w_0}$ is a simple tilting module in $\mathcal{O}^{\mathfrak{p}}_0$.
Applying to $L_{w'_0w_0}$ indecomposable projective functors outputs either
zero or indecomposable tilting modules in $\mathcal{O}^{\mathfrak{p}}_0$.
By construction, the corresponding (graded) split Grothendieck group
is isomorphic to $\underline{\hat{H}}_{w'_0w_0}\mathbf{H}$.

Also, by construction, $\mathbf{RN}_{w'_0w_0}$ is isomorphic to the (graded)
Grothendieck group of $\mathcal{D}^b(\mathcal{O}^{\mathfrak{p}}_0)$.
Since tilting modules generate $\mathcal{D}^b(\mathcal{O}^{\mathfrak{p}}_0)$,
it follows that $\mathbf{RN}_{w'_0w_0}=\underline{\hat{H}}_{w'_0w_0}\mathbf{H}$.
This implies our claim.
\end{proof}

Note that $w'_0w_0= w_0(w_0w'_0w_0)$ and that conjugation by $w_0$
corresponds to an automorphism of the Dynkin diagram. This means that
$w_0w'_0w_0$ is the longest element of some parabolic subgroup of $W$
and hence  
$\mathbf{H}\underline{\hat{H}}_{w'_0w_0}=\mathbf{LN}_{w'_0w_0}$
is a consequence of Proposition~\ref{prop-s3.4-4}.

\subsection{Motivation for the study of cyclic modules: 
combinatorial K{\aa}hrstr{\"o}m's condition}\label{s3.9}

Our main motivation for the study of cyclic submodules generated by the 
elements of the (dual) Kazhdan-Lusztig basis stems from the following:
In \cite[Conjecture~1.2]{KMM}, for an involution
$w\in S_n$, the following condition appears in the
context of the study of Kostant's problem: whether there exist
different $x,y\in S_n$ such that both
$\underline{\hat{H}}_w\underline{H}_x\neq 0$
and $\underline{\hat{H}}_w\underline{H}_y\neq 0$
as well as $\underline{\hat{H}}_w\underline{H}_x=\underline{\hat{H}}_w\underline{H}_y$.
The formulation of this is due to Johan K{\aa}hrstr{\"o}m, so will
call this condition the {\em (graded) combinatorial K{\aa}hrstr{\"o}m's condition}
for $w$. We will denote by $\mathbf{Kh}(w)\in\{\mathtt{true},
\mathtt{false}\}$ the logical value of the claim that this
condition is satisfied. The corresponding {\em ungraded condition} 
is obtained from the graded one by evaluating $v$ to $1$
(in particular, the ungraded condition is, potentially, weaker,
however, one of the expectations of \cite[Conjecture~1.2]{KMM}
is that the two conditions are, in fact, equivalent).

The condition appears as a conjecturally equivalent condition for a negative
answer to Kostant's problem for $L_w$, see \cite{Jo}.
It is known (see \cite{MS1,MS2}) that the answer to Kostant's problem is
an invariant of a KL left cell. Therefore, it is natural to ask whether
the combinatorial K{\aa}hrstr{\"o}m's conditions are left cell invariant.
Rewriting $\underline{\hat{H}}_w\underline{H}_x=\underline{\hat{H}}_w\underline{H}_y$
as $\underline{\hat{H}}_w(\underline{H}_x-\underline{H}_y)=0$, we transform the
condition into a property that some particular element belongs to the
right annihilator of $\underline{\hat{H}}_w$. Clearly, if this is the case,
then the same element belongs to the right annihilator of any element from
the left cyclic submodule $\mathbf{H}\underline{\hat{H}}_w$. 
Therefore it is important to understand
the elements contained in $\mathbf{H}\underline{\hat{H}}_w$.

We discuss K{\aa}hrstr{\"o}m's conjecture and condition in more
{{}{}detail} in the next section.

\section{K{\aa}hrstr{\"o}m's conditions}\label{s4}

\subsection{Kostant's problem}\label{s4.1}

Let $U(\mathfrak{g})$ be the universal enveloping algebra of
the Lie algebra $\mathfrak{g}$. For $w\in W$, consider $L_w$ and the algebra 
$\mathrm{End}_{\mathbb{C}}(L_w)$ of all $\mathbb{C}$-linear 
operators on the underlying vector space of $L_w$. The vector space
$\mathrm{End}_{\mathbb{C}}(L_w)$ has the natural structure of a
$U(\mathfrak{g})$-$U(\mathfrak{g})$-bimodule given by the action
of $U(\mathfrak{g})$ on the two arguments of
$\mathrm{Hom}_{\mathbb{C}}(L_w,L_w)=\mathrm{End}_{\mathbb{C}}(L_w)$.
The representation map defines a bimodule homomorphism from 
$U(\mathfrak{g})$ to $\mathrm{End}_{\mathbb{C}}(L_w)$.

Denote by $\mathcal{L}(L_w,L_w)$ the maximal subbimodule of 
$\mathrm{End}_{\mathbb{C}}(L_w)$, the adjoint action of 
$U(\mathfrak{g})$ on which is locally finite. Then the image
of $U(\mathfrak{g})$ in $\mathrm{End}_{\mathbb{C}}(L_w)$ 
under the representation map is contained in $\mathcal{L}(L_w,L_w)$.
The classical {\em Kostant's problem} for $L_w$, see \cite{Jo},
asks whether this map from $U(\mathfrak{g})$ to $\mathcal{L}(L_w,L_w)$
is surjective. This is open, in general. For partial answers and 
some recent progress, we refer to \cite{Ma23,MSr,KMM,MMM,CM1,CM2,Ka,KaM}.

\subsection{ K{\aa}hrstr{\"o}m's conjecture}\label{s4.2}

The following conjecture, which is due to Johan K{\aa}hr\-str{\"o}m,
is taken from \cite[Conjecture~1.2]{KMM}:

\begin{conjecture}[K{\aa}hrstr{\"o}m's conjecture]\label{Kh-conj}
For an involution $d\in S_n$, the following conditions are equivalent:

\begin{enumerate}[$($A$)$]
\item\label{Kh-conj-0} The representation map 
$U(\mathfrak{g})\to \mathcal{L}(L_d,L_d)$ is surjective.
\item\label{Kh-conj-1} There do not exist $x,y\in W$ such that
$x\neq y$ and $\theta_xL_d\cong \theta_y L_d\neq 0$ 
in $\mathcal{O}_0$.
\item\label{Kh-conj-2} There do not  exist $x,y\in W$ such that
$x\neq y$ and $\theta_xL_d\cong \theta_y L_d\neq 0$ 
in ${}^\mathbb{Z}\mathcal{O}_0$.
\item\label{Kh-conj-3} There  do not exist $x,y\in W$ such that
$x\neq y$ and $\underline{\hat{H}}_d\underline{H}_x\vert_{{}_{v=1}}=
\underline{\hat{H}}_d\underline{H}_y\vert_{{}_{v=1}}\neq 0$.
\item\label{Kh-conj-4} There  do not exist $x,y\in W$ such that
$x\neq y$ and $\underline{\hat{H}}_d\underline{H}_x=
\underline{\hat{H}}_d\underline{H}_y\neq 0$.
\end{enumerate}
We will use the following terminology: we will call 
\begin{itemize}
\item the negation of the condition in Conjecture~\ref{Kh-conj}\eqref{Kh-conj-1}
the {\em categorical ungraded K{\aa}hr\-str{\"o}m's condition},
\item the negation of the  condition in Conjecture~\ref{Kh-conj}\eqref{Kh-conj-2}
the {\em categorical graded K{\aa}hr\-str{\"o}m's condition},
\item the negation of the  condition in Conjecture~\ref{Kh-conj}\eqref{Kh-conj-3}
the {\em combinatorial ungraded K{\aa}hr\-str{\"o}m's condition},
\item the negation of the  condition in Conjecture~\ref{Kh-conj}\eqref{Kh-conj-4}
the {\em combinatorial graded K{\aa}hr\-str{\"o}m's condition}.
\end{itemize}
\end{conjecture}

One can, of course, extend the formulation of all conditions in
Conjecture~\ref{Kh-conj} from involutions to all elements of 
the symmetric group. By \cite{MS1,MS2}, the condition 
in Conjecture~\ref{Kh-conj}\eqref{Kh-conj-0} is an invariant of
a KL left cell. This raises the natural question whether all the
other conditions are invariants of KL left cells as well.

We also note the following obvious implications between the
various conditions in Conjecture~\ref{Kh-conj}:
\eqref{Kh-conj-1}$\Rightarrow$\eqref{Kh-conj-2},
\eqref{Kh-conj-3}$\Rightarrow$\eqref{Kh-conj-4},
\eqref{Kh-conj-3}$\Rightarrow$\eqref{Kh-conj-1},
\eqref{Kh-conj-4}$\Rightarrow$\eqref{Kh-conj-2}.
Also, from \cite[Theorem~B]{KMM}, it follows that
\eqref{Kh-conj-0}$\Rightarrow$\eqref{Kh-conj-1}
and \eqref{Kh-conj-0}$\Rightarrow$\eqref{Kh-conj-2}.
In fact, \cite[Theorem~B]{KMM} implies that 
\eqref{Kh-conj-0} is equivalent to the conjunction of 
\eqref{Kh-conj-2} with the condition that 
$\theta_x L_d$ is either indecomposable or zero, for all $x\in W$.
The latter is unknown in general and is formulated as a conjecture
in \cite[Conjecture~2]{KiM}, called the {\em Indecomposability Conjecture}.
Note that the answer to this conjecture is an invariant of a
KL left cell, that is, for $w,w'\in W$ such that $w\sim_L w'$, 
the module $\theta_x L_w$ is either indecomposable or zero, 
for all $x\in W$, if and only if the module 
$\theta_x L_{w'}$ is either indecomposable or zero, 
for all $x\in W$, see \cite[Proposition~2]{CMZ}.

\subsection{Graded vs ungraded categorical conditions}\label{s4.3}

We start with the following observation.
 
\begin{proposition}\label{prop-s4.3-1}
Assume that the Indecomposability Conjecture is true for some element
$w\in W$, that is, the module $\theta_x L_w$ is either indecomposable or zero, 
for all $x\in W$. Then Condition~\eqref{Kh-conj-2}
implies Condition~\eqref{Kh-conj-1}.
\end{proposition}

\begin{proof}
We prove that the negation of Condition~\eqref{Kh-conj-1} implies the negation of 
Condition~\eqref{Kh-conj-2}.

Assume that $x,y\in W$ are different and  $\theta_x L_w\cong \theta_y L_w\neq 0$
in $\mathcal{O}_0$. As both $L_w$ and $\theta_x$ admit graded lifts, so does
$\theta_x L_w$ (and hence $\theta_y L_w$ as well). As $\theta_x L_w$
is, by assumption, indecomposable, its graded lift is unique, up to isomorphism
and shift of grading, see \cite[Lemma~1.5]{St0}. As $L_x$, by convention, is concentrated
in degree zero and $\theta_x$ commutes with the graded 
lift of the duality, the module $\theta_x L_w$ is graded self-dual. 
Therefore the ungraded isomorphism $\theta_x L_w\cong \theta_y L_w$
lifts to a graded isomorphism as the graded self-duality of the 
two non-zero involved modules does not allow for any shift of grading.
\end{proof}

\subsection{Left cell invariance of the categorical  
K{\aa}hrstr{\"o}m's conditions}\label{s4.4}

Next we establish the left cell invariance of the graded 
categorical condition.

\begin{proposition}\label{prop-s4.4-1}
Let $w,w'\in S_n$ be such that $w\sim_L w'$.
If there are different $x,y\in S_n$ such that
$\theta_x L_w\cong \theta_y L_w\neq 0$ in
${}^\mathbb{Z}\mathcal{O}_0$ (or in $\mathcal{O}_0$), then 
$\theta_x L_{w'}\cong \theta_y L_{w'}\neq 0$ in
${}^\mathbb{Z}\mathcal{O}_0$ (resp. in $\mathcal{O}_0$, and note that 
for the same $x$ and $y$).
\end{proposition}

\begin{proof}
We prove the ungraded version of the claim. The graded version follows by
using the standard graded lifts.

In order to prove our claim, we need to enlarge our universe to 
that of thick category $\mathcal{O}$, alternatively, the
category of Harish-Chandra bimodules, see \cite{So92}.
So, we consider the category ${}^\infty_{\hspace{1mm}0}\mathcal{H}^\infty_0$
which consists of all finitely generated 
$U(\mathfrak{g})$-$U(\mathfrak{g})$-bimodules $M$
which satisfy the following conditions:
\begin{itemize}
\item the adjoint action of $\mathfrak{g}$ on $M$ is
locally finite and has finite multiplicities for all
simple finite dimensional $\mathfrak{g}$-modules;
\item the kernel $\mathbf{m}$ of the central character of the trivial
$\mathfrak{g}$-module acts on $M$ locally finitely both 
on the left and on the right.
\end{itemize}
Inside ${}^\infty_{\hspace{1mm}0}\mathcal{H}^\infty_0$
we have the full subcategories 
${}^1_{0}\mathcal{H}^\infty_0$ and
${}^\infty_{\hspace{1mm}0}\mathcal{H}^1_0$
defined by the conditions that their objects 
are annihilated  by $\mathbf{m}$ on the left
or right, respectively. Note that these two categories are
equivalent via a natural map that changes the sides of
the bimodule, see \cite{So86}. The latter map is 
a self-equivalence of ${}^\infty_{\hspace{1mm}0}\mathcal{H}^\infty_0$.

The categories $\mathcal{O}_0$ and
${}^\infty_{\hspace{1mm}0}\mathcal{H}^1_0$
are equivalent, see \cite[Theorem~5.9]{BG}.
We also have the bicategories
$\tilde{\mathscr{P}}^l$ and $\tilde{\mathscr{P}}^r$
of left and right projective endofunctors of 
${}^\infty_{\hspace{1mm}0}\mathcal{H}^\infty_0$.
The two action are swapped under the above 
self-equivalence of 
${}^\infty_{\hspace{1mm}0}\mathcal{H}^\infty_0$
given by changing the sides of the bimodules.
We also point out the obvious phenomenon of the
situation that the actions of 
$\tilde{\mathscr{P}}^l$ and $\tilde{\mathscr{P}}^r$
on ${}^\infty_{\hspace{1mm}0}\mathcal{H}^\infty_0$
functorially commute. This is due to the fact that
they act on two different
sides of a bimodule.

The action of $\tilde{\mathscr{P}}^l$ preserves 
${}^\infty_{\hspace{1mm}0}\mathcal{H}^1_0$
and, under the equivalence between 
${}^\infty_{\hspace{1mm}0}\mathcal{H}^1_0$
and $\mathcal{O}_0$, corresponds to the action of
$\mathscr{P}$ on $\mathcal{O}_0$. Indecomposable
objects in both $\tilde{\mathscr{P}}^l$ and 
$\tilde{\mathscr{P}}^r$ are indexed by $W$
and we will use the notation $\theta_w^l$
and $\theta_w^r$ for them, respectively.
We emphasize one significant confusion of the situation:
the action of $\mathscr{P}$ on $\mathcal{O}_0$
is a right action with respect to our indexing
and it corresponds to the action of $\tilde{\mathscr{P}}^l$
which acts on the left of a bimodule. 

The categories ${}^\infty_{\hspace{1mm}0}\mathcal{H}^1_0$,
$\tilde{\mathscr{P}}^l$ and 
$\tilde{\mathscr{P}}^r$ all admit graded lifts, see \cite{So92}.
We identify $\mathcal{O}_0$ and
${}^\infty_{\hspace{1mm}0}\mathcal{H}^1_0$ under the
equivalence from \cite[Theorem~5.9]{BG}. In particular,
all simple objects of ${}^\infty_{\hspace{1mm}0}\mathcal{H}^1_0$
become simple objects in $\mathcal{O}_0$ in this way.

Now, let us assume that $x,y\in S_n$ are different and
such that $\theta_x^l L_w\cong \theta_y^l L_w\neq 0$ in
%${}^\mathbb{Z}\mathcal{O}_0$.
$\mathcal{O}_0$. For any $\theta^r\in \tilde{\mathscr{P}}^r$,
we thus have $(\theta_x^l L_w)\theta^r\cong (\theta_y^l L_w)\theta^r$
and, since the left and the right actions commute, we have
$\theta_x^l (L_w \theta^r)\cong \theta_y^l (L_w \theta^r)$.

Let $\mathcal{Q}$ be the Serre subcategory of 
${}^1_{0}\mathcal{H}^\infty_0$ generated by all objects
of the form $L_w \theta^r$, where $\theta^r\in \tilde{\mathscr{P}}^r$.
If we go through all $a\in S_n$ in the two-sided cell of $w$,
then the additive closure of the corresponding
objects $L_w \theta^r_a$ is exactly the category of 
projective-injective objects in $\mathcal{Q}$, see \cite[Theorem~6]{Ma2}.
Furthermore, as we are in type $A$, for each 
$a\in S_n$ in the two-sided cell of $w$,
the corresponding $L_w \theta^r_a$ is either indecomposable or zero,
see \cite[Theorem~43]{MM1}. Moreover, by the same result, all
homomorphisms between these indecomposable projective-injective
objects are obtained by evaluating at $L_w$ some homomorphisms
between the corresponding projective functors.

Now take our $w'\in S_n$ such that $w\sim_L w'$ and consider a
complex
\begin{displaymath}
\mathcal{X}:\qquad \dots \to X_{-2}\to X_{-1}\to X_0\to 0 
\end{displaymath}
of projective-injective objects in $\mathcal{Q}$ with the
following properties:
\begin{itemize}
\item $X_0$ is the projective cover of $L_{w'}$;
\item all morphisms in this complex are radical morphisms;
\item if $u\in S_n$ is such that $u\sim_L w$ and
$L_u$ appears as a subquotient of some homology
of $\mathcal{X}$, then it appears in homological position
$0$, moreover, $u=w'$ and this $L_u$ corresponds to the
top of $X_0$.
\end{itemize}
From the uniqueness of a minimal projective resolution it
follows that these properties define the complex 
$\mathcal{X}$ uniquely, up to isomorphism.
For the graded version it is important to point out that
$\mathcal{X}$ is gradeable.

Now recall that each indecomposable constituent of 
$\mathcal{X}$ is of the form $L_w \theta^r_a$, for some
$a\in S_n$, and each homomorphism between such modules
comes from evaluation of some homomorphism between the
corresponding projective functors. Also, the actions
of both $\theta_x^l$ and $\theta_y^l$ functorially commute
with the action of $\tilde{\mathscr{P}}^r$. Consequently,
it follows that the complexes $\theta_x^l(\mathcal{X})$
and $\theta_y^l(\mathcal{X})$ are isomorphic.
In particular, the zero homologies of these two 
complexes are isomorphic.

It remains to note that the zero homology of 
$\theta_x^l(\mathcal{X})$ surjects onto 
$\theta_x^lL_{w'}$. This is true 
by construction as $X_0$ is the projective cover of $L_{w'}$
and all morphisms in $\mathcal{X}$ are radical morphisms.
The kernel of this surjection
is maximal with respect to the property that it does not
contain any simple subquotients of the form
$L_b$, where $b\sim_J w$. Similarly for $\theta_y^l(\mathcal{X})$.
Consequently, $\theta_x^lL_{w'}\cong \theta_y^lL_{w'}$.
\end{proof}

\subsection{Combinatorial K{\aa}hrstr{\"o}m's conditions are compatible with
parabolic induction}\label{s4.9}

Let $W'$ be a parabolic subgroup of $W$ and $w'_0$ the longest element of $W'$.
Let $\mathfrak{g}'$ be the Lie algebra corresponding to $W'$. Then, for
$w\in W'$, \cite[Theorem~1.1]{Ka} asserts that $L_w$ is Kostant positive
for $\mathfrak{g}'$ if and only if $L_{ww'_0w_0}$ is Kostant positive for $\mathfrak{g}$.
In this subsection, we prove an analogous result for 
K{\aa}hrstr{\"o}m's combinatorial conditions.

\begin{proposition}\label{prop-s5.4-1}
Let $W'$ be a parabolic subgroup of $W$ and $w\in W'$. Then the following
two conditions are equivalent:
\begin{enumerate}[$($a$)$]
\item\label{prop-s5.4-1.1} There exist different $x,y\in W'$ such that 
$[\theta_x L_w]=[\theta_{y} L_w]\neq 0$ in $\mathbf{Gr}(\mathcal{O}_0)$
for $\mathfrak{g}'$.
\item\label{prop-s5.4-1.2} There exist different $a,b\in W$ such that 
$[\theta_a L_{ww'_0w_0}]=[\theta_b L_{ww'_0w_0}]\neq 0$ in 
$\mathbf{Gr}(\mathcal{O}_0)$
for $\mathfrak{g}$.
\end{enumerate}
\end{proposition}

{{}{}
To prove this proposition, we will need some preparation.
Denote by $\mathtt{R}$ the set of shortest representatives
for cosets in $w_0{W'}w_0\setminus W$. 
For $c\in \mathtt{R}$, consider the Serre subcategory 
$\mathcal{X}_c$ of $\mathcal{O}_0$ for $\mathfrak{g}$
generated by all $L_u$ such that $u\in W'w'_0w_0c$.
In particular, $\mathcal{X}=\mathcal{X}_e$.

\begin{lemma}\label{lem-s5.4-2}
For $v\in W'$ and $c\in \mathtt{R}$, in $\mathbf{Gr}(\mathcal{O}_0)$, we have 
\begin{equation}\label{lem-s5.4-2-eq}
[\theta_c L_{vw'_0w_0}]\in
[L_{vw'_0w_0c}]+\sum_{c>c'\in \mathtt{R}}\,\,
\sum_{u\in w_0W'w_0} \mathbb{Z}_{\geq 0}[L_{uc'}].
\end{equation}
\end{lemma}

\begin{proof}
Consider first the case $v=w'_0$. In this case $vw'_0w_0=w_0$, so we are talking about
$L_{w_0}=\Delta_{w_0}=\nabla_{w_0}=T_{w_0}$ and its image under 
$\theta_c$. This image is exactly the indecomposable tilting module
$T_{w_0c}$ which contains $\nabla_{w_0c}$ (with multiplicity $1$) 
in its costandard  filtration and all other dual Verma subquotients 
in this filtration are of the form 
$\nabla_h$, for some $h>w_0c$. 
Since $w_0c$ is a shortest representative in $W'w_0c$, it follows that
\begin{equation}\label{lem-s5.4-2-eq2}
[T_{w_0c}]\in
[\nabla_{w_0c}]+\sum_{c>c'\in \mathtt{R}}\sum_{u\in w_0W'w_0} \mathbb{Z}_{\geq 0}[L_{uc'}],
\end{equation}
which implies Formula~\eqref{lem-s5.4-2-eq} in the case $v=w'_0$.

Now we apply to $\theta_cL_{w_0}$ the twisting functor $\top_h$, where $h\in W'$.
As it commutes with projective functors, we get $\top_h \theta_cL_{w_0}\cong
\theta_c\top_h L_{w_0}$. Here $\top_h L_{w_0}\cong\nabla_{hw_0}$ and therefore
we are talking about the module $\theta_c \nabla_{hw_0}$.
On the other hand, if we look at the module $\top_h T_{w_0c}$
and consider Formula~\eqref{lem-s5.4-2-eq2}, we obtain
\begin{equation}\label{lem-s5.4-2-eq3}
[\theta_c\nabla_{hw_0}]=
[\top_h T_{w_0c}]\in
[\nabla_{hw_0c}]+\sum_{c>c'\in \mathtt{R}}\sum_{u\in w_0W'w_0} \mathbb{Z}_{\geq 0}[L_{uc'}].
\end{equation}
Note that, for $h,h'\in W'$, we have
\begin{displaymath}
[\nabla_{hw'_0w_0c}:L_{h'w'_0w_0c}]=[\nabla_{hw'_0w_0}:L_{h'w'_0w_0}], 
\end{displaymath}
see \cite[Theorem~37]{CMZ}. Therefore Formula~\eqref{lem-s5.4-2-eq} follows
now in the general case by induction on $\ell(v)$.
\end{proof}

\begin{lemma}\label{lem-s5.4-3}
For $v\in W$, the inequality $\theta_v L_{ww'_0w_0}\neq 0$ implies that
there is $c\in \mathtt{R}$ and $d\in W'$ such that $v=w_0w'_0dw'_0w_0c$
and $\theta_{d} L_w\neq 0$.
\end{lemma}

\begin{proof}
Each $v\in W$ can be written uniquely as $w_0w'_0dw'_0w_0c$, for some
$c\in \mathtt{R}$ and $d\in W'$. Note that $\theta_{w_0w'_0dw'_0w_0c}$
is a summand of $\theta_c\circ \theta_{w_0w'_0dw'_0w_0}$. Therefore
$\theta_{d} L_w=0$, which is equivalent to  
$\theta_{w_0w'_0dw'_0w_0} L_{ww'_0w_0}=0$, implies $\theta_c\circ
\theta_{w_0w'_0dw'_0w_0} L_{ww'_0w_0}=0$ and thus 
$\theta_{w_0w'_0dw'_0w_0c} L_{ww'_0w_0}=0$ as well.
The claim follows.
\end{proof}

\begin{lemma}\label{lem-s5.4-7}
If $f,d\in W'$ and $c\in\mathtt{R}$ are
such that $\theta_{w_0w'_0fw'_0w_0c}$ is a summand of 
$\theta_c\circ\theta_{w_0w'_0dw'_0w_0}$, then $f=d$.
\end{lemma}

\begin{proof}
The fact that $\theta_{w_0w'_0fw'_0w_0c}$ is a summand of 
$\theta_c\circ\theta_{w_0w'_0dw'_0w_0}$ is equivalent to the fact that
$\theta_{(w_0w'_0fw'_0w_0c)^{-1}}$ is a summand of 
$\theta_{(w_0w'_0dw'_0w_0)^{-1}}\circ\theta_{c^{-1}}$. The latter
is equivalent to $L_{c^{-1}}$ being a subquotient
of $\theta_{w_0w'_0dw'_0w_0} L_{c^{-1}(w_0w'_0fw'_0w_0)^{-1}}$.

We know that, for any $m,n\in W$, the module $L_e$ is a subquotient of
$\theta_m L_n$ if and only if $m=n^{-1}$. From
\cite[Theorem~37]{CMZ} it now follows that 
$L_{c^{-1}}$ is a subquotient of 
$\theta_{w_0w'_0dw'_0w_0} L_{c^{-1}(w_0w'_0fw'_0w_0)^{-1}}$
if and only if $d=f$. This completes the proof.
\end{proof}
}

\begin{proof}[Proof of Proposition~\ref{prop-s5.4-1}]
Consider the Serre subcategory $\mathcal{X}$ of $\mathcal{O}_0$ for $\mathfrak{g}$
generated by all $L_u$ such that $u\in W'w_0$. The category 
$\mathcal{X}$ is equivalent to the category $\mathcal{O}_0$ for $\mathfrak{g}'$ 
such that, for $v\in W'$, the object $L_v$ of the latter is sent to 
the object $L_{vw'_0w_0}$ of the former. Moreover, this equivalence 
intertwines the action of the projective functor $\theta_v$ 
on $\mathcal{O}_0$ for $\mathfrak{g}'$ with  the action of the projective functor
$\theta_{w_0w'_0vw'_0w_0}$ on $\mathcal{X}$,
cf. \cite[Theorem~37]{CMZ}. Consequently, for the implication
\eqref{prop-s5.4-1.1}$\Rightarrow$\eqref{prop-s5.4-1.2},
we can just take $a=w_0w'_0xw'_0w_0$ and $b=w_0yw_0$. 

It remains to prove the opposite implication 
\eqref{prop-s5.4-1.2}$\Rightarrow$\eqref{prop-s5.4-1.1}.
{{}{}Assume} that there 
{{}{}exist} different elements $a,b\in W$ having the property that
$[\theta_a L_{ww'_0w_0}]=[\theta_b L_{ww'_0w_0}]\neq 0$ in $\mathbf{Gr}(\mathcal{O}_0)$
for $\mathfrak{g}$. Write $a=w_0w'_0dw'_0w_0c$ and also $b=w_0w'_0d'w'_0w_0c'$, for
some $c,c'\in \mathtt{R}$ and $d,d'\in W'$. Note that
$\theta_a$ is a summand of $\theta_c\circ \theta_{w_0w'_0dw'_0w_0}$ and
$\theta_b$ is a summand of $\theta_{c'}\circ \theta_{w_0w'_0d'w'_0w_0}$.
From Lemma~\ref{lem-s5.4-2}, it follows that $c=c'$ and,
form Lemma~\ref{lem-s5.4-3}, we have
$\theta_{w_0w'_0dw'_0w_0}L_{ww'_0w_0}\neq 0$
and $\theta_{w_0w'_0d'w'_0w_0}L_{ww'_0w_0}\neq 0$.

Assume that $[\theta_{w_0w'_0dw'_0w_0}L_{ww'_0w_0}]\neq [\theta_{w_0w'_0d'w'_0w_0}L_{ww'_0w_0}]$.
Applying $\theta_c$ and using Lemma~\ref{lem-s5.4-2}, we obtain that there
is $k\in W'$ such that
\begin{equation}\label{lem-s5.4-eqeq1} 
[\theta_c\circ\theta_{w_0w'_0dw'_0w_0}L_{ww'_0w_0}:L_k]\neq 
[\theta_c\circ\theta_{w_0w'_0d'w'_0w_0}L_{ww'_0w_0}:L_k]. 
\end{equation}
From the combination of Lemma~\ref{lem-s5.4-7} with Lemma~\ref{lem-s5.4-2},
we see that the only summand of $\theta_c\circ\theta_{w_0w'_0dw'_0w_0}$
which may have a non-zero contribution to 
Formula~\eqref{lem-s5.4-eqeq1}  is the summand $\theta_{w_0w'_0d'w'_0w_0c}$.
Therefore, Formula~\eqref{lem-s5.4-eqeq1} implies
\begin{displaymath}
[\theta_{w_0w'_0dw'_0w_0c}L_{ww'_0w_0}]\neq [\theta_{w_0w'_0d'w'_0w_0c}L_{ww'_0w_0}], 
\end{displaymath}
a contradiction.
 
Thus $[\theta_{w_0w'_0dw'_0w_0}L_{ww'_0w_0}]=[\theta_{w_0w'_0d'w'_0w_0}L_{ww'_0w_0}]$ 
and hence $[\theta_d L_w]=[\theta_{d'}L_w]$, so we can take $x=d$
and $y=d'$ to get condition \eqref{prop-s5.4-1.1}.
\end{proof}

\begin{proposition}\label{prop-s5.5-1}
Let $W'$ be a parabolic subgroup of $W$ and $w\in W'$. Then the following
two conditions are equivalent:
\begin{enumerate}[$($a$)$]
\item\label{prop-s5.5-1.1} There exist different $x,y\in W'$ such that 
$\theta_x L_w\neq 0$ and $[\theta_x L_w]=[\theta_{y} L_w]$ in 
$\mathbf{Gr}({}^\mathbb{Z}\mathcal{O}_0)$
for $\mathfrak{g}'$.
\item\label{prop-s5.5-1.2} There exist $a,b\in W$ with $a\neq b$,
such that we have $\theta_a L_{ww'_0w_0}\neq 0$ as well as
$[\theta_a L_{ww'_0w_0}]=[\theta_b L_{ww'_0w_0}]$ in 
$\mathbf{Gr}({}^\mathbb{Z}\mathcal{O}_0)$
for $\mathfrak{g}$.
\end{enumerate}
\end{proposition}

\begin{proof}
Mutatis mutandis the proof of Proposition~\ref{prop-s5.4-1}. 
\end{proof}

\subsection{The elements $x$ and $y$  in K{\aa}hrstr{\"o}m's conditions
must be in the same left cell}\label{s4.5}

\begin{proposition}\label{prop-s4.5-1}
{\hspace{1mm}}

\begin{enumerate}[$($a$)$]
\item\label{prop-s4.5-1.1} Let $x,y,z\in S_n$ be such that 
$\theta_x L_z\cong \theta_y L_z\neq 0$ in $\mathcal{O}_0$. Then 
$x\sim_L y$.
\item\label{prop-s4.5-1.2} Let $x,y,z\in S_n$ be such that 
$\theta_x L_z\cong \theta_y L_z\neq 0$ in ${}^\mathbb{Z}\mathcal{O}_0$. 
Then  $x\sim_L y$.
\item\label{prop-s4.5-1.3} Let $x,y,z\in S_n$ be such that 
$(\underline{\hat{H}}_z\underline{H}_x)\vert_{{}_{v=1}}=
(\underline{\hat{H}}_z\underline{H}_y)\vert_{{}_{v=1}}\neq 0$. Then 
$x\sim_L y$.
\item\label{prop-s4.5-1.4} Let $x,y,z\in S_n$ be such that 
$\underline{\hat{H}}_z\underline{H}_x=
\underline{\hat{H}}_z\underline{H}_y\neq 0$. Then 
$x\sim_L y$.
\end{enumerate}
\end{proposition}

{{}{}
To prove this proposition, we will need the following lemma.

\begin{lemma}\label{lem-s4.5-2}
Let $a,b\in W$ be such that 
$\underline{\hat{H}}_a\underline{H}_b\neq 0$.
When expressing $\underline{\hat{H}}_a\underline{H}_b$
in the KL basis, some $\underline{H}_c$ with $c\sim_L b$
will appear with non-zero coefficient.
\end{lemma}

\begin{proof}
From the definition of the left order it follows that any 
$\underline{H}_f$ which appears with a non-zero coefficient
when expressing $\underline{\hat{H}}_a\underline{H}_b$
in the KL basis satisfies $f\geq_L b$. 

Consider the graded module $\theta_b L_a$ and some minimal
projective resolution $\mathcal{P}_\bullet$ of this module.
Then, in $\mathbf{Gr}({}^\mathbb{Z}\mathcal{O}_0)$, 
we have the following equality:
\begin{equation}\label{eq-s4.5-3}
\underline{\hat{H}}_a\underline{H}_b= 
[\theta_b L_a]=
\sum_{i\in\mathbb{Z}}(-1)^i[\mathcal{P}_i].
\end{equation}
Here each $[\mathcal{P}_i]$ is a sum of $[P_f]=\underline{H}_f$,
up to some power of $v$, that is, up to some shift of grading.
Let $k$ be the projective dimension of $\theta_b L_a$,
in particular, $\mathcal{P}_k\neq 0$.
Then, by \cite[Theorem~A]{KMM}, each summand
of $\mathcal{P}_k$ is of the form $P_f$, for some $f\sim_L b$.
Unfortunately, this is not enough as the corresponding $[P_f]$
might cancel in the alternating sum in Equation~\eqref{eq-s4.5-3}.
We need to prove that this will not happen for at least some $f$.
We have $\mathrm{Ext}^k(\theta_bL_a,L_f)\neq 0$ which we
consider ungraded, so far.

For this we need to go through the proof of \cite[Theorem~A]{KMM}
with a fine-tooth comb. Let $d$ be the Duflo involution in the 
left cell of $b$. Then the proof of \cite[Theorem~A]{KMM}
shows that $\mathrm{Ext}^k(\theta_bL_a,L_f)\neq 0$
implies $\mathrm{Ext}^k(\theta_bL_a,\theta_dL_f)\neq 0$
(in fact, just for this implication the degree $k$ can be replaced
by any degree). Both modules $\theta_bL_a$ and $\theta_dL_f$
can be represented, in the derived category, by linear 
complexes of tilting modules, say $\mathcal{T}^1_\bullet$
and $\mathcal{T}^2_\bullet$, respectively.

Since tilting modules are {{}{} self-orthogonal
in the homological sense}, extensions between
$\theta_bL_a$ and $\theta_dL_f$ can now be computed in the 
homotopy category of complexes of titling modules. In fact,
the non-zero extension of degree $k$ in $\mathrm{Ext}^k(\theta_bL_a,\theta_dL_f)$
in the proof of \cite[Theorem~A]{KMM} is constructed by 
lifting the identity map on some indecomposable 
summand, call it $N$, that is common for the rightmost 
non-zero component of $\mathcal{T}^2_\bullet$ and the 
leftmost non-zero component of $\mathcal{T}^1_\bullet$.
This gives rise to a non-zero homomorphism from
$\mathcal{T}^1_\bullet$ to $\mathcal{T}^2_\bullet\langle -k\rangle[k]$.

We claim that, in the homotopy category of complexes of tilting
modules, the hom-space from $\mathcal{T}^1_\bullet$ to
$\mathcal{T}^2_\bullet\langle -k\rangle[m]$ is zero, for any
$m\neq k$. Indeed, if $m>k$, the claim is obvious as $k$
is the projective dimension of $\theta_b L_a$. If $m<k$, we
use the fact that both $\mathcal{T}^2_\bullet$ and
$\mathcal{T}^1_\bullet$ are linear complexes. 
Let $i$ be some homological position. Then any 
indecomposable summand $T_r$ of $\mathcal{T}^1_i$
is shifted by $-i$, while any summand $T_s$ of
$\mathcal{T}^2_{-i+m}\langle -k\rangle[m]$ is shifted
by $-(i+(k-m))$ and $i+(k-m)>i$. By (graded) Ringel
self-duality of $\mathcal{O}_0$, the endomorphism algebra
of the characteristic tilting module in $\mathcal{O}_0$
is isomorphic to $A$ as a graded algebra. In particular,
it is positively graded. So, there are no non-zero 
homomorphisms from $T_r\langle-i\rangle$ to
$T_s\langle-(i+(k-m))\rangle$. This implies our claim.

Our claim yields the equality
$\mathrm{ext}^i(\theta_bL_a,\theta_dL_f\langle -k\rangle)=0$
which, in turn, implies the equality 
$\mathrm{ext}^i(\theta_bL_a,L_f\langle -k\rangle)=0$.
Therefore the corresponding $[P_f]$ summand cannot be cancelled
and the claim of the lemma follows.
\end{proof}

}

\begin{proof}[Proof of Proposition~\ref{prop-s4.5-1}]
As $\theta_x $, $\theta_y$ and $L_z$ are all gradeable, 
so are $\theta_x L_z$ and $\theta_y L_z$. This means that 
any ungraded isomorphism between $\theta_x L_z$ and 
$\theta_y L_z$ can be lifted to a graded isomorphism, 
up to shift of grading.
However, since both $\theta_x L_z$ and $\theta_x L_z$
are self-dual with respect to the graded duality
(since the standard lifts of simple modules and the standard lifts
of indecomposable projective functors are), no shift of
grading in the graded isomorphism of $\theta_x L_z$ and 
$\theta_y L_z$ is necessary. Consequently,
Claim~\eqref{prop-s4.5-1.1} follows from 
Claim~\eqref{prop-s4.5-1.2}. We prove Claim~\eqref{prop-s4.5-1.2}.

By Koszul-Ringel self-duality, see \cite[Theorem~16]{Ma2}, we have 
$\theta_x L_z\cong \theta_y L_z\neq 0$ if and only if
$\theta_{z^{-1}w_0} L_{w_0x^{-1}}\cong \theta_{z^{-1}w_0} L_{w_0y^{-1}}\neq 0$.
Any socle and any top constituent of the non-zero module
$\theta_{z^{-1}w_0} L_{w_0x^{-1}}$ is of the form $L_a$, for some
$a$ in the right KL-cell of $w_0x^{-1}$, see e.g. \cite[Corollary~14]{MM1}. 
Similarly for $\theta_{z^{-1}w_0} L_{w_0y^{-1}}$. This implies that 
$w_0x^{-1}\sim_R w_0y^{-1}$ which, in turn, implies $x\sim_L y$
using \cite[Corollary~6.2.10]{BB}, completing the proof of 
Claim~\eqref{prop-s4.5-1.2}.

It is worth to emphasize that the above proof uses non-combinatorial 
tools (like Koszul-Ringel duality as well as top and socle). These are
not available for Claims~\eqref{prop-s4.5-1.3} and \eqref{prop-s4.5-1.4}.
As we will see, this will make the corresponding proofs significantly
more complicated. 

Let us now prove Claim~\eqref{prop-s4.5-1.4}. If
$\underline{\hat{H}}_z\underline{H}_x\neq 0$, then,
by Lemma~\ref{lem-s4.5-2}, some 
$\underline{H}_c$ with {{}{}$c\sim_L x$} appears
with a non-zero coefficient when
expressing $\underline{\hat{H}}_z\underline{H}_x$
in the KL basis. Similarly for 
$\underline{\hat{H}}_z\underline{H}_y$, which 
implies $x\sim_L y$ and proves Claim~\eqref{prop-s4.5-1.4}.

To prove Claim~\eqref{prop-s4.5-1.3}, we write $[\theta_xL_z]$
in the twisted KL basis $\{[T_a]\,:\,a\in S_n\}$. To do this, we use 
the Koszul-Ringel self-duality, see \cite[Theorem~16]{Ma2}.
It represents $\theta_xL_z$ via the homology at the 
homological position $0$ of a certain linear complex 
$\mathcal{T}_\bullet$  of tilting  modules. The category
of linear complexes of tilting modules is equivalent to
${}^\mathbb{Z}\mathcal{O}_0$, and, under this equivalence,
the complex $\mathcal{T}_\bullet$ is mapped to 
$\theta_{z^{-1}w_0} L_{w_0x^{-1}}$. The module 
$\theta_{z^{-1}w_0} L_{w_0x^{-1}}$ has the following two
properties:
\begin{itemize}
\item the parity condition: simple modules appearing 
in the same homogeneous component are indexed by 
the elements of $W$ of the same parity, and this parity
alternates with the degree of the homogeneous component;
\item the two extreme non-zero homogeneous components 
contain only simples indexed by the elements from 
the same right cell as $w_0x^{-1}$;
\item all other non-zero homogeneous components 
contain only simples indexed by the elements 
that are right less than or equivalent to $w_0x^{-1}$.
\end{itemize}
The parity condition implies that, when we write 
$[\theta_xL_z]$ as  $\displaystyle \sum_{i\in\mathbb{Z}}
(-1)^i[\mathcal{T}_i]$, there will be no cancellation 
between the summands as the summands indexed by the elements
of the same parity will always appear with the same sign.
The second condition means that we will necessarily have
some $[T_{a}]$, with $w_0a^{-1}\sim_R w_0x^{-1}$, appearing
with a non-zero coefficient (moreover, this coefficient is
a Laurent polynomial with non-negative coefficients
and hence does not vanish under the specialization $v=1$).
But this means that
$a\sim_L x$. Also, all other $[T_{b}]$ which appear with 
non-zero coefficient satisfy $w_0b^{-1}\leq_R w_0x^{-1}$.
But this means that $b\geq_L x$.

Since we also have a similar property for $y$, it follows
that $x\sim_L y$. This completes the proof.
\end{proof}

We now present an alternative proof of (c) and (d) in 
Proposition~\ref{prop-s4.5-1}. {{}{}We think
this alternative argument worth being recorded as it 
approaches the claim from a purely combinatorial perspective.}
For $x,y,w\in S_{n}$, 
we denote by $m_{w,x}^{y}, n_{x,w}^{y}\in\mathbb{A}$ the elements such that
\begin{equation}\label{Eq:DKLPositiveCoeffs}
\hat{\underline{H}}_{w}\underline{H}_{x}=\sum_{y\leq_{R}w}m_{w,x}^{y}\hat{\underline{H}}_{y}, \hspace{1mm} \text{ and } \hspace{2mm} \underline{H}_{x}\hat{\underline{H}}_{w}=\sum_{y\leq_{L}w}n_{x,w}^{y}\hat{\underline{H}}_{y}.
\end{equation}
The conditions $y\leq_{R}w$ and $y\leq_{L}w$ in the above summations follow from Section~\ref{s3.2}. Since both $\hat{\underline{H}}_{w}\underline{H}_{x}=[\theta_{x}L_{w}]$ and $\hat{\underline{H}}_{y}=[L_{y}]$, the coefficients $m_{w,x}^{y}$ record the graded multiplicity of $L_{y}$ within $\theta_{x}L_{w}$. Thus we, in fact, have $m_{w,x}^{y}\in\mathbb{A}_{\geq0}:=\mathbb{Z}_{\geq0}[v^{\pm}]$. Acting on the former equation above by $(-)^{\ast}$, we see that $n_{x,w}^{y}\in\mathbb{A}_{\geq0}$ as well.

We start with a few technical lemmata.

\begin{lemma}\label{Lem:StarDKL}
For any $w\in W$, we have that $(\hat{\underline{H}}_{w})^{\ast}=\hat{\underline{H}}_{w^{-1}}$.
\end{lemma}

\begin{proof}
The dual KL basis element $\hat{\underline{H}}_{w}$ is uniquely defined by the collection of equations
\[ \boldsymbol{\tau}(\hat{\underline{H}}_{w}\underline{H}_{x^{-1}})=\boldsymbol{\tau}(\underline{H}_{x^{-1}}\hat{\underline{H}}_{w})=\delta_{w,x}, \] 
for all $x\in W$. By Lemma~\ref{lem-InvBasicProps}, applying $(-)^{\ast}$ to the arguments of $\boldsymbol{\tau}$ above, we see that
\[ \boldsymbol{\tau}(\underline{H}_{x}(\hat{\underline{H}}_{w})^{\ast})=\boldsymbol{\tau}((\hat{\underline{H}}_{w})^{\ast}\underline{H}_{x})=\delta_{w,x}=\delta_{w^{-1},x^{-1}}, \]
for all $x\in W$. Therefore, the element $(\hat{\underline{H}}_{w})^{\ast}$ satisfies the collection of equations which uniquely define $\hat{\underline{H}}_{w^{-1}}$, and thus these two elements must agree.
\end{proof}

\begin{lemma}\label{Lem:KLConjHw0}
For any $w\in W$, we have that $H_{w_{0}}\underline{H}_{w}H_{w_{0}}^{-1}=\underline{H}_{w_{0}ww_{0}}$.
\end{lemma}

\begin{proof}
By Equation~\eqref{eq-virk}, we have that
\begin{equation}\label{Eq:VF1}
\beta(\hat{\underline{H}}_{w})=\underline{H}_{ww_{0}}H_{w_{0}}.
\end{equation}
By (ii) of Lemma~\ref{lem-InvBasicProps}, the involutions $(-)^{\ast}$ and $\boldsymbol{\beta}$ commute. Therefore, by Lemma~\ref{Lem:StarDKL}, applying $(-)^{\ast}$ to both sides of Equation~\eqref{Eq:VF1} gives us the equation
\[ \boldsymbol{\beta}(\hat{\underline{H}}_{w^{-1}})=H_{w_{0}}\underline{H}_{w_{0}w^{-1}}. \]
Since $w$ is arbitrary, we may replace $w^{-1}$ with $w$, giving us the equation
\begin{equation}\label{Eq:VF2}
\boldsymbol{\beta}(\hat{\underline{H}}_{w})=H_{w_{0}}\underline{H}_{w_{0}w},
\end{equation}
which holds for all $w\in W$. Equating the right-hand sides of both Equation~\eqref{Eq:VF1} and Equation~\eqref{Eq:VF2}, and then multiplying by $H_{w_{0}}^{-1}$ 
{{}{}on} the right, gives the equation
\[ \underline{H}_{ww_{0}}=H_{w_{0}}\underline{H}_{w_{0}w}H_{w_{0}}^{-1}, \]
which holds for all $w\in W$. Again, since $w$ is arbitrary, we can replace $w$ with $w_{0}w$, and the resulting relation is precisely the statement of the lemma.
\end{proof}

\begin{lemma}\label{Lem:DKLConjHw0}
For any $w\in W$, we have that $H_{w_{0}}\hat{\underline{H}}_{w}H_{w_{0}}^{-1}=\hat{\underline{H}}_{w_{0}ww_{0}}$.
\end{lemma}

\begin{proof}
By Lemma~\ref{Lem:KLConjHw0}, we have that
\begin{align*}
&H_{w_{0}}\underline{H}_{w}H_{w_{0}}^{-1}=\underline{H}_{w_{0}ww_{0}} \\
\iff &H_{w_{0}}\left(\boldsymbol{\beta}(\hat{\underline{H}}_{ww_{0}})H_{w_{0}}^{-1}\right)H_{w_{0}}^{-1}=\boldsymbol{\beta}(\hat{\underline{H}}_{w_{0}w})H_{w_{0}}^{-1} \\
\iff &H_{w_{0}}\boldsymbol{\beta}(\hat{\underline{H}}_{ww_{0}})H_{w_{0}}^{-1}=\boldsymbol{\beta}(\hat{\underline{H}}_{w_{0}w}) \\
\iff &H_{w_{0}}\hat{\underline{H}}_{ww_{0}}H_{w_{0}}^{-1}=\hat{\underline{H}}_{w_{0}w},
\end{align*}
where the first equivalence follows by applying Equation~\eqref{eq-virk} on both sides, the second by multiplying by $H_{w_{0}}$ on the right, and the third by applying the automorphism $\boldsymbol{\beta}$. Since this holds for all $w\in W$, replacing $w$ with $ww_{0}$ in the final equation above produces the relation given in the statement of the lemma, and so we are done.
\end{proof}

\begin{proposition}\label{Prop:EqImplyRC}
Given $x,y,z\in S_{n}$, 
we have the implication
\[ \hat{\underline{H}}_{x}\underline{H}_{z}=\hat{\underline{H}}_{y}\underline{H}_{z}\neq0 \implies x\sim_{R}y.  \]
\end{proposition}

\begin{proof}
By Equation~\eqref{Eq:DKLPositiveCoeffs} we have that
\[ \hat{\underline{H}}_{x}\underline{H}_{z}=\sum_{a\leq_{R}x}m_{x,z}^{a}\hat{\underline{H}}_{a}. \]
Moreover, by \cite[Corollary~14]{MM1}, we know that if $L_{a}$ belongs to the top or socle of $\theta_{z}L_{x}$, then $a\sim_{R}x$. Since $m_{x,z}^{a}$ corresponds to the graded multiplicity of $L_{a}$ in $\theta_{z}L_{x}$, it must be the case that $m_{x,z}^{a}\neq0$ for some $a\sim_{R}x$, or in other words, $\hat{\underline{H}}_{a}$ appears with non-zero coefficient in $\hat{\underline{H}}_{x}\underline{H}_{z}$ when expressed in terms of the dual KL basis. Since $\hat{\underline{H}}_{x}\underline{H}_{z}=\hat{\underline{H}}_{y}\underline{H}_{z}$, we must also have that $\hat{\underline{H}}_{a}$ appears with non-zero coefficient in $\hat{\underline{H}}_{y}\underline{H}_{z}$ when expressed in terms of the dual KL basis. This implies that $x\sim_{R}a\leq_{R}y$, and by symmetry in $x$ and $y$, we also have that $y\leq_{R}x$, and hence $x\sim_{R}y$. 
\end{proof}

\begin{proposition}\label{Prop:EqImplyLC}
Given $x,y,z\in S_{n}$, then we have the implication
\[ \hat{\underline{H}}_{z}\underline{H}_{x}=\hat{\underline{H}}_{z}\underline{H}_{y}\neq0 \implies x\sim_{L}y.  \]
\end{proposition}

\begin{proof}
We have the following equivalences:
\begin{align*}
\hat{\underline{H}}_{z}\underline{H}_{x}=\hat{\underline{H}}_{z}\underline{H}_{y}&\neq0 \iff \\
\boldsymbol{\beta}(\underline{H}_{zw_{0}})H_{w_{0}}\boldsymbol{\beta}(\hat{\underline{H}}_{xw_{0}})H_{w_{0}}^{-1}=\boldsymbol{\beta}(\underline{H}_{zw_{0}})H_{w_{0}}\boldsymbol{\beta}(\hat{\underline{H}}_{yw_{0}})H_{w_{0}}^{-1}&\neq0 \iff \\
\underline{H}_{zw_{0}}H_{w_{0}}\hat{\underline{H}}_{xw_{0}}H_{w_{0}}^{-1}=\underline{H}_{zw_{0}}H_{w_{0}}\hat{\underline{H}}_{yw_{0}}H_{w_{0}}^{-1}&\neq0 \iff \\
\underline{H}_{zw_{0}}\hat{\underline{H}}_{w_{0}x}=\underline{H}_{zw_{0}}\hat{\underline{H}}_{w_{0}y}&\neq0 \iff \\
\hat{\underline{H}}_{x^{-1}w_{0}}\underline{H}_{w_{0}z^{-1}}=\hat{\underline{H}}_{y^{-1}w_{0}}\underline{H}_{w_{0}z^{-1}}&\neq0.
\end{align*}
The first equivalence above follows by applying Equation~\eqref{eq-virk}, the second by applying $\boldsymbol{\beta}$, the third by applying Lemma~\ref{Lem:DKLConjHw0}, and the fourth by applying the anti-automorphism $({}_-)^{\ast}$ and by Lemma~\ref{Lem:StarDKL}. By Proposition~\ref{Prop:EqImplyRC}, the last line above implies $x^{-1}w_{0}\sim_{R}y^{-1}w_{0}$, which is equivalent to $x^{-1}\sim_{R}y^{-1}$ and, in turn, $x\sim_{L}y$.
\end{proof}

The above proposition is precisely (d) of Proposition~\ref{prop-s4.5-1}. To obtain (c), one needs to note that both of the above results hold in the ungraded setting, and are proved in the exact same manner by simply replacing each involution and basis element with their corresponding ungraded counterpart. That is, making the following replacements:
\[ \boldsymbol{\tau}\rightsquigarrow\boldsymbol{\tau}_{v=1}, \hspace{2mm} (-)^{\ast}\rightsquigarrow(-)_{v=1}^{\ast}, \hspace{2mm} \boldsymbol{\beta}\rightsquigarrow\boldsymbol{\beta}_{v=1}, \]  
\[ H_{w}\rightsquigarrow H_{w}|_{{}_{v=1}}, \hspace{2mm} \underline{H}_{w}\rightsquigarrow \underline{H}_{w}|_{{}_{v=1}}, \hspace{2mm} \hat{\underline{H}}_{w}\rightsquigarrow \hat{\underline{H}}_{w}|_{{}_{v=1}}. \] 
Also, for Proposition~\ref{Prop:EqImplyRC}, one further needs to replace $m_{w,x}^{y}$ with $m_{w,x}^{y}|_{{}_{v=1}}$, with the latter recording the usual multiplicity of $L_{y}$ in the module $\theta_{x}L_{w}$. As such, we have that $m_{w,x}^{y}\neq0$ if and only if $m_{w,x}^{y}|_{{}_{v=1}}\neq0$, allowing the application of \cite[Corollary~14]{MM1} in the proof of Proposition~\ref{Prop:EqImplyRC} to remain valid in the ungraded case.

\subsection{Varying $x$ and $y$ in the categorical 
K{\aa}hrstr{\"o}m's conditions}\label{s4.6}

\begin{proposition}\label{prop-s4.6-1}
Let $x,y,w\in S_n$ be such that $x\neq y$ and 
$\theta_x L_w\cong \theta_y L_w\neq 0$. Let $\tilde{x},\tilde{y}\in S_n$ be 
such that $x\sim_R \tilde{x}$, $y\sim_R \tilde{y}$ and $\tilde{x}\sim_L \tilde{y}$.
Then $\theta_{\tilde{x}} L_w\cong \theta_{\tilde{y}} L_w$.
\end{proposition}

\begin{proof}
It is enough to prove the Ringel-Koszul dual claim that 
\begin{equation}\label{eq-s4.6-2}
\theta_{w^{-1}w_0}L_{w_0\tilde{x}^{-1}}\cong \theta_{w^{-1}w_0}L_{w_0\tilde{y}^{-1}}. 
\end{equation}
From Proposition~\ref{prop-s4.5-1} we know that $x\sim_L y$
and hence $w_0{x}^{-1}\sim_R w_0{y}^{-1}$.
Let $R$ be the right cell containing both 
$w_0{x}^{-1}$ and $w_0{y}^{-1}$.
Let $\tilde{R}$ be the right cell containing both 
$w_0\tilde{x}^{-1}$ and $w_0\tilde{y}^{-1}$.
By \cite[Proposition~35]{MS1} or \cite[Theorem~43]{MM1}, 
the cell $2$-representations of $\mathscr{P}$ corresponding
to $R$ and $\tilde{R}$ are equivalent. This equivalence 
preserves the $\mathcal{L}$-relation, that is, it maps
the simple object corresponding to $w_0{x}^{-1}$
to the simple object corresponding to $w_0{\tilde{x}}^{-1}$
and the simple object corresponding to $w_0{y}^{-1}$
to the simple object corresponding to $w_0{\tilde{y}}^{-1}$. 

We denote the simple object of the cell $2$-representation
corresponding to $w_0{{x}}^{-1}$ by $M$,
the simple object of the cell $2$-representation
corresponding to $w_0{{y}}^{-1}$ by $N$,
the simple object corresponding to $w_0{\tilde{x}}^{-1}$
by $\tilde{M}$ and
the simple object corresponding to $w_0{\tilde{y}}^{-1}$
by $\tilde{N}$. Then $M$ is obtained from 
$L_{w_0{{x}}^{-1}}$ by a partial coapproximation functor
with respect to the projective injective objects of 
the cell $2$-representation, see \cite[Subsection~2.5]{KM}.
This partial coapproximation functor 
commutes with projective functors (as the latter obviously
preserve projective injective objects). Similarly,
such a  partial coapproximation also sends
$L_{w_0{{y}}^{-1}}$ to $N$, then 
$L_{w_0{\tilde{x}}^{-1}}$ to $\tilde{M}$ 
and, finally, $L_{w_0{\tilde{y}}^{-1}}$ to $\tilde{N}$.

Now, from 
$\theta_{w^{-1}w_0}L_{w_0{x}^{-1}}\cong 
\theta_{w^{-1}w_0}L_{w_0{y}^{-1}}$
(which is the Koszul dual of 
our assumption $\theta_x L_w\cong \theta_y L_w$), 
it follows that
$\theta_{w^{-1}w_0} M\cong \theta_{w^{-1}w_0} N$
using our partial coapproximation and then that 
$\theta_{w^{-1}w_0} \tilde{M}\cong \theta_{w^{-1}w_0} \tilde{N}$
via our equivalence. Realized
as an object in category $\mathcal{O}$, the module $\tilde{M}$ has simple top
$L_{w_0{\tilde{x}}^{-1}}$ with all other composition 
subquotients indexed by elements of strictly smaller right cells.
Similarly, $\tilde{N}$ has simple top
$L_{w_0{\tilde{y}}^{-1}}$ with all other composition 
subquotients indexed by elements of strictly smaller right cells.
Therefore $\theta_{w^{-1}w_0} M\cong \theta_{w^{-1}w_0} N$
implies \eqref{eq-s4.6-2} by the same arguments as in the 
proof of Proposition~\ref{prop-s4.4-1}, completing the proof.
\end{proof}

\subsection{Towards the left cell invariance of combinatorial 
K{\aa}hrstr{\"o}m's conditions}\label{s4.7}

We expect the following:

\begin{conjecture}\label{conj-4.7-1}
Let $x,y,w,w'\in S_n$ be such that 
$\underline{\hat{H}}_w\underline{H}_x=\underline{\hat{H}}_w\underline{H}_y\neq 0$
and $w\sim_L w'$.
Then $\underline{\hat{H}}_{w'}\underline{H}_x=\underline{\hat{H}}_{w'}\underline{H}_y$
in $\mathbf{H}$.
\end{conjecture}

\begin{conjecture}\label{conj-4.7-2}
Let $x,y,w,w'\in S_n$ be such that 
$\underline{\hat{H}}_w\underline{H}_x\vert_{{}_{v=1}}=
\underline{\hat{H}}_w\underline{H}_y\vert_{{}_{v=1}}\neq 0$
and $w\sim_L w'$.
Then $\underline{\hat{H}}_{w'}\underline{H}_x\vert_{{}_{v=1}}
=\underline{\hat{H}}_{w'}\underline{H}_y\vert_{{}_{v=1}}$.
\end{conjecture}

Below we give some thoughts and evidence that support these 
conjectures.

\begin{proposition}\label{prop-4.7-3}
Let $x,y,w,w'\in S_n$ be such that 
$\underline{\hat{H}}_w\underline{H}_x=\underline{\hat{H}}_w\underline{H}_y\neq 0$
and, additionally, $w\sim_L w'$. Then
$\underline{\hat{H}}_{w'}(\underline{H}_x-\underline{H}_y)$
belongs to the $\mathbb{A}$-linear span of 
$\underline{\hat{H}}_a$, where $a<_R w'$.
\end{proposition}

\begin{proof}
Recall that the quotient of the $\mathbb{A}$-linear span of 
$\underline{\hat{H}}_a$, where $a\leq_R w$, by the $\mathbb{A}$-linear span of 
$\underline{\hat{H}}_a$, where $a<_R w$, is called the 
right cell $\mathbf{H}$-module associated to the right KL cell
$\mathcal{R}_w$ of $w$ and denoted $\mathbf{C}_{\mathcal{R}_w}$. 
Similarly, we have the right KL cell $\mathbf{H}$-module
$\mathbf{C}_{\mathcal{R}_{w'}}$ associated to the right KL cell
$\mathcal{R}_{w'}$ of $w'$. In type $A$, the modules 
$\mathbf{C}_{\mathcal{R}_w}$ and 
$\mathbf{C}_{\mathcal{R}_{w'}}$ are isomorphic, since 
$w\sim_L w'$. Moreover, such an isomorphism
maps the class of $\underline{\hat{H}}_w$
to the class of $\underline{\hat{H}}_{w'}$.

Since this is an isomorphism of right $\mathbf{H}$-modules, it commutes
with the right action of $\mathbf{H}$ which exactly gives the 
claim of the proposition. The appearance of the
uncontrollable term from 
the $\mathbb{A}$-linear span of 
$\underline{\hat{H}}_a$, where $a<_R w'$, in the formulation 
is due to the fact that
the latter span is zero in $\mathbf{C}_{\mathcal{R}_{w'}}$.
\end{proof}

For an element $u$ in the
$\mathbb{Z}$-linear span of the elements of the dual
KL basis and a simple reflection $s$, we write
$u=u_s^+ +u_s^-$, where $u_s^+$ combines all 
terms of $u$ corresponding to $\underline{\hat{H}}_a$
with $sa<a$, while $u_s^-$ combines all 
terms of $u$ corresponding to $\underline{\hat{H}}_a$
with $sa>a$.

\begin{proposition}\label{prop-4.7-4}
Assume the following property: for any $x,y\in S_n$,
any simple reflection $s$, and any $u$ in the
$\mathbb{Z}$-linear span of the elements of the dual
KL basis, the equality $u(\underline{H}_x-\underline{H}_y)=0$
implies $u_s^+(\underline{H}_x-\underline{H}_y)=0$.
Then Conjecture~\ref{conj-4.7-1} is true.
\end{proposition}

\begin{proof}
Denote by $K$ the left $\mathbf{H}$-submodule of ${}_\mathbf{H}\mathbf{H}$
consisting of all elements $u$ such that $u(\underline{H}_x-\underline{H}_y)=0$.
For $u\in \mathbf{H}$, we can uniquely write 
$u=\sum_{i\in\mathbb{Z}}u_i$, where each $u_i$ is $v^i$
times a $\mathbb{Z}$-linear span of the elements of the dual
KL basis.

We claim that, under the assumption of our proposition,
given $u\in \mathbf{H}$ such that each
$u_i$ is in $K$, any $u'\in \mathbf{H}u$ has all $u'_i\in K$.
Indeed, it is enough to prove this for the elements of the form
$\underline{H}_au$, since the KL basis is a basis. Furthermore,
due to the recursive nature of the KL basis with respect to
multiplication, it is enough to prove this for elements of the form
$\underline{H}_su$, where $s$ is a simple reflection. Finally,
it is enough to prove the claim for each $\underline{H}_su_i$,
that is, under the assumption that $u$ itself is 
a $\mathbb{Z}$-linear span of the elements of the dual
KL basis.

In the latter case, $\underline{H}_su$ equals
$u_s^+$ with coefficient $v+v^{-1}$ plus some $u'$
which is again a $\mathbb{Z}$-linear span of the elements of the dual
KL basis. Due to the assumption of our proposition,
$(v+v^{-1})u_s^+\in K$. Since $\underline{H}_su\in K$,
it follows that $\underline{H}_su-(v+v^{-1})u_s^+\in K$.
This completes the proof of our claim.

Now, if $\underline{\hat{H}}_w(\underline{H}_x-\underline{H}_y)=0$,
we can find $a\in S_n$ in the same two-sided KL cell as $w$
such that $u=\underline{H}_a\underline{\hat{H}}_w$ has
$\underline{\hat{H}}_{w'}$ as one of the $u_i$. This implies
$\underline{\hat{H}}_{w'}(\underline{H}_x-\underline{H}_y)=0$
by our claim above.
\end{proof}

\begin{proposition}\label{prop-4.7-5}
Let $x,y,w,w'\in S_n$ be such that 
$\underline{\hat{H}}_w\underline{H}_x=\underline{\hat{H}}_w\underline{H}_y\neq 0$
and, additionally, $w\sim_L w'$. Then there is an element
$u$ in the $\mathbf{A}$-linear span of the elements of the form
$\underline{\hat{H}}_a$, where $a<_Rw$, such that 
$(\underline{\hat{H}}_{w'}+u)\underline{H}_x=
(\underline{\hat{H}}_{w'}+u)\underline{H}_y$
in $\mathbf{H}$.
\end{proposition}

\begin{proof}
As mentioned in the proof of Proposition~\ref{prop-4.7-3}, there is
an isomorphism of cell modules which maps the class of 
$\underline{\hat{H}}_w$ to the class of $\underline{\hat{H}}_{w'}$.
The additional summand $u$ appears as the cell module is the quotient
by the $\mathbf{A}$-linear span of the elements of the form
$\underline{\hat{H}}_a$, where $a<_Rw$.
\end{proof}

\begin{proposition}\label{prop-4.7-6}
Let $x,y,w,w'\in S_n$ be such that 
$\underline{\hat{H}}_w\underline{H}_x\vert_{{}_{v=1}}
=\underline{\hat{H}}_w\underline{H}_y\vert_{{}_{v=1}}\neq 0$
and $w\sim_L w'$. Then there is an element
$u$ in the $\mathbf{A}$-linear span of the elements of the form
$\underline{\hat{H}}_a$, where $a<_Rw$, such that 
$(\underline{\hat{H}}_{w'}+u)\underline{H}_x\vert_{{}_{v=1}}=
(\underline{\hat{H}}_{w'}+u)\underline{H}_y\vert_{{}_{v=1}}$
in $\mathbf{H}$.
\end{proposition}

\begin{proof}
The proof of Proposition~\ref{prop-4.7-5} works
mutatis mutandis also after the evaluation of $v$ to $1$.
\end{proof}

\subsection{Some other conjectures}\label{s4.8}

We expect the following:

\begin{conjecture}\label{conj-4.8-1}
Let $x,y,w,x',y'\in S_n$ be such that 
$\underline{\hat{H}}_w\underline{H}_x=\underline{\hat{H}}_w\underline{H}_y\neq 0$,
$x\sim_R x'$, $y\sim_R y'$ and $x'\sim_L y'$.
Then $\underline{\hat{H}}_{w}\underline{H}_{x'}=
\underline{\hat{H}}_{w}\underline{H}_{y'}$
in $\mathbf{H}$.
\end{conjecture}

\begin{conjecture}\label{conj-4.8-2}
Let $x,y,w,x',y'\in S_n$ be such that 
$\underline{\hat{H}}_w\underline{H}_x\vert_{{}_{v=1}}=
\underline{\hat{H}}_w\underline{H}_y\vert_{{}_{v=1}}\neq 0$,
$x\sim_R x'$, $y\sim_R y'$ and $x'\sim_L y'$.
Then $\underline{\hat{H}}_{w}\underline{H}_{x'}\vert_{{}_{v=1}}
=\underline{\hat{H}}_{w}\underline{H}_{y'}\vert_{{}_{v=1}}$.
\end{conjecture}

\begin{proposition}\label{prop-4.8-3}
Let $x,y,w,x',y'\in S_n$ be such that 
$\underline{\hat{H}}_w\underline{H}_x=\underline{\hat{H}}_w\underline{H}_y\neq 0$,
$x\sim_R x'$, $y\sim_R y'$ and $x'\sim_L y'$.
Then there is an element $u$ in the $\mathbf{A}$-linear span of the 
elements of the form $\underline{{H}}_a$, where $a>_Rx$ and $a>_Ry$, such that 
$\underline{\hat{H}}_{w}(\underline{H}_{x'}-\underline{H}_{y'}+u)=0$
in $\mathbf{H}$.
\end{proposition}

\begin{proof}
Similarly to the proof of  Proposition~\ref{prop-4.7-5}, multiplying 
$\underline{H}_x-\underline{H}_y$ with elements of $\mathbf{H}$,
we can hit  $\underline{H}_{x'}-\underline{H}_{y'}$, up to 
an uncontrollable linear combination of KL basis
elements  that are strictly right KL-bigger than both $x$ and $y$.
\end{proof}

\begin{proposition}\label{prop-4.8-4}
Let $x,y,w,x',y'\in S_n$ be such that 
$\underline{\hat{H}}_w\underline{H}_x\vert_{{}_{v=1}}
=\underline{\hat{H}}_w\underline{H}_y\neq 0\vert_{{}_{v=1}}$,
$x\sim_R x'$, $y\sim_R y'$ and $x'\sim_L y'$.
Then there is an element $u$ in the $\mathbf{A}$-linear span of the 
elements of the form $\underline{{H}}_a$, where $a>_Rx$ and $a>_Ry$, such that 
we have
$\underline{\hat{H}}_{w}(\underline{H}_{x'}-\underline{H}_{y'}+u)\vert_{{}_{v=1}}$.
\end{proposition}

\begin{proof}
The proof of Proposition~\ref{prop-4.8-4} works
mutatis mutandis also after the evaluation of $v$ to $1$.
\end{proof}

\vspace{2mm}

\noindent
Department of Mathematics, Uppsala University, Box. 480,
SE-75106, Uppsala, \\ SWEDEN, 
emails:
{\tt samuel.creedon\symbol{64}math.uu.se}\hspace{5mm}
{\tt mazor\symbol{64}math.uu.se}

\end{document}